\documentclass[a4paper,10pt,reqno]{amsart}
\usepackage{amssymb,bbm}
\usepackage{array}
\usepackage{mathtools}
\usepackage{graphicx}
\usepackage[top=1.5in, bottom = 1.5in, left = 1in, right = 1in]{geometry}
\newcolumntype{L}{>{\centering\arraybackslash}m{1.5cm}}
\usepackage{amsthm, calrsfs}
\usepackage[mathscr]{eucal}
\usepackage{rotating}
\usepackage{multirow,float}
\usepackage{color}
\numberwithin{equation}{section}
\newtheorem{theorem}{Theorem}
\newtheorem{lemma}{Lemma}
\newtheorem{proposition}{Proposition}
\newtheorem{corollary}{Corollary}

\theoremstyle{definition}
\newtheorem{definition}[theorem]{Definition}

\newtheorem{example}[theorem]{Example}
\theoremstyle{remark}
\usepackage[colorlinks=true]{hyperref}

\usepackage{enumerate}

\newcommand{\RN}[1]{\uppercase\expandafter{\romannumeral#1}}

\begin{document}
	\title[Memory Dependent Growth in Sublinear Volterra Differential Equations]{Memory Dependent Growth in Sublinear Volterra Differential Equations}
	\author{John A. D. Appleby}
	\address{School of Mathematical
		Sciences, Dublin City University, Glasnevin, Dublin 9, Ireland}
	\email{john.appleby@dcu.ie} \urladdr{webpages.dcu.ie/\textasciitilde applebyj}
	
	\author{Denis D. Patterson}
	\address{School of Mathematical
		Sciences, Dublin City University, Glasnevin, Dublin 9, Ireland}
	\email{denis.patterson2@mail.dcu.ie} \urladdr{sites.google.com/a/mail.dcu.ie/denis-patterson}
	
	\thanks{Denis Patterson is supported by the Government of Ireland Postgraduate Scholarship Scheme operated by the Irish Research Council under the project GOIPG/2013/402.} 
	\date{\today}
\begin{abstract}
We investigate memory dependent asymptotic growth in scalar Volterra  equations with sublinear nonlinearity. To obtain precise results we utilise the powerful theory of regular variation extensively. By computing the growth rate in terms of a related ordinary differential equation we show that when the memory effect is so strong that the kernel tends to infinity, the growth rate of solutions depends explicitly on the memory of the system. Finally, we employ a fixed point argument to determine analogous results for a perturbed Volterra equation and show that, for a sufficiently large perturbation, the solution tracks the perturbation asymptotically, even when the forcing term is potentially highly non-monotone.
\end{abstract}
\maketitle
\section{Introduction}
We investigate explicit memory dependence in the asymptotic growth rates of positive solutions of the following scalar Volterra integro-differential equation
\begin{align}\label{functional}
x'(t) &= \int_{[0,t]} \mu(ds)f(x(t-s)),\quad t > 0; \quad x(0) = \xi>0,
\end{align}
where $f$ is a positive sublinear function (i.e. $\lim_{x\to\infty}f(x)/x=0$) and $\mu$ is a non--negative Borel measure. The relevant existence and uniqueness theory regarding equations of the form \eqref{functional} is well known and guarantees a unique solution $x \in C(\mathbb{R}^+;(0,\infty))$ in the framework of this article \cite[Corollary 12.3.2]{GLS}, with the convention that $\mathbb{R}^+:= [0,\infty)$. By defining the function 
\begin{align}\label{M_defn}
M(t):= \int_{[0,t]} \mu(ds), \quad t \geq 0,
\end{align}
it follows that \eqref{functional} is equivalent to 
\begin{align}\label{volterra}
x(t) &= x(0) + \int_0^t M(t-s)f(x(s))\,ds,\quad t \geq 0, \quad x(0)=\xi>0.
\end{align}
We also study the asymptotic behaviour of the perturbed Volterra equation  
\begin{align}\label{functional_pert}
x'(t) &= \int_{[0,t]} \mu(ds)f(x(t-s)) + h(t),\quad t > 0; \quad x(0) = \xi>0.
\end{align}
As with the unperturbed equation, it is useful to consider an integral form of 
\eqref{functional_pert}, and by defining 
\begin{equation} \label{def.H}
H(t) := \int_0^t h(s)\,ds, \quad t\geq 0,
\end{equation} 
it follows that \eqref{functional_pert} can be written in integral form as
\begin{align}\label{volterra_pert}
x(t) &= x(0) + \int_0^t M(t-s)f(x(s))\,ds + H(t),\quad t \geq 0; \quad x(0)=\xi>0.
\end{align}

In \cite{sublinear2015}, with $\mu$ a finite measure, we demonstrate that when $f$ is sublinear and asymptotically increasing, the solution of \eqref{functional} obeys 
$
\lim_{t\to\infty}F(x(t))/t = \int_{[0,\infty)}\mu(ds) < \infty,
$
where 
\begin{align}\label{cap_F}
F(x) := \int_1^x \frac{1}{f(u)}\,du, \quad x >0.
\end{align}
In other words, the structure of the memory does not affect the asymptotic growth rate of the solution of \eqref{functional} when the total measure is finite: indeed, the entire mass of $\mu$ could be concentrated at $0$, because the ordinary differential equation $y'(t)=\int_{[0,\infty)} \mu(ds) \cdot f(y(t))$ for $t\geq 0$ also obeys $F(y(t))/t\to \int_{[0,\infty)} \mu(ds)$ as $t\to\infty$. This is in contrast to the linear case where the growth rate depends crucially on the structure of the memory (cf. \cite[Theorem 7.2.3]{GLS}). In \cite{sublinear2015} we also show that if $\lim_{t\to\infty}M(t) = \infty$, then $\lim_{t\to\infty}F(x(t))/t = \infty$. This result suggests that allowing the total measure to be infinite makes the long run dynamics more sensitive to the memory but that comparison with a non-autonomous ordinary differential equation may be necessary in this case.

To achieve precise asymptotic results for the solutions of \eqref{functional} and \eqref{functional_pert} we employ the theory of regular variation extensively. We record now for the reader's convenience the definition of a regularly varying function (in the sense of Karamata) and allied notation. 
\begin{definition}
	Suppose a measurable function $h:\mathbb{R}\to (0,\infty)$ obeys
	\[
	\lim_{x\to\infty}\frac{h(\lambda\,x)}{h(x)} = \lambda^\rho, \mbox{ for all } \lambda>0, \mbox{ some }\rho\in\mathbb{R},
	\]
	then $h$ is \emph{regularly varying at infinity} with index $\rho$, or $h\in \text{RV}_\infty(\rho)$.
\end{definition}
Regular variation provides a natural generalisation of the class of power functions and the application of the theory of regular variation to the study of qualitative properties of differential equations is an active area of investigation. Recent research themes in this direction are recorded in reviews such as \cite{maric2000regular} and \cite{rehakrv} and all properties of regularly varying functions employed can be found in the classic text \cite{BGT}. The authors of the present paper give a highly abridged list of the properties we have found useful in  the introduction to our work~\cite{appleby2013classification}, which concerns ordinary differential equations.  

Many of the applications of regular variation in the asymptotic theory of \emph{linear} Volterra equations deal with the situation in which it is desired to capture slow decay in the memory, as captured by a measure or kernel, or a singularity. Of course, slowly fading memory can be described in other ways, using for instance the theory of $L^1$ weighted spaces (see e.g.,\cite{shea1975variants} and for stochastic equations, \cite{applebyweighted2010}). When the kernel is integrable, it is often possible to obtain precise rates of decay in $L^\infty$ by means of a larger class of kernels (such as the subexponential class studied in \cite{applebyexactrates06}, of which regularly varying kernels are a subclass). However, for singular equations, or equations with non--integrable kernels, the full power of the theory of regular variation is 
often needed: in particular, for linear equations, transform methods and the Abelian and Tauberian theorems for regular variation are exploited (see e.g.,\cite{appleby2011long,wong1974asymptotic}). It should be stressed, though, that such methods are of greatest utility for linear equations: indeed, there does not seem to be especial benefit gained in this work in applying such a transform approach. Moreover, in this paper, the equation is intrinsically non--linear: $f(x)$ is not of linear order as $x\to\infty$, and regular variation arises \emph{both} in the slow decay of $\mu$ and in the sublinear growth of $f$. Also, it is a general theme of the works cited above that the slow decay in the memory, combined with an appropriate type of stability, give rise to convergence at a certain rate to equilibrium. By contrast in this paper, solutions grow, rather than decay. 

With a view to applications, we believe the most interesting subclass of equations will retain the property that the asymptotic contribution to the growth rate from a moving interval of any fixed duration ($\tau>0$, say) is negligible, in the sense that 
\begin{equation} \label{eq.fademu}
\lim_{t\to\infty} \int_{[t,t+\tau)}\mu(ds)=0, \quad \text{ for each $\tau>0$}.
\end{equation}
It should be noted that our proofs \emph{do not} require this stipulation, but we mention it in order to motivate shortly a stronger hypothesis on $M$.  

With \eqref{eq.fademu} still in force, if $\mu$ is absolutely continuous and admits a non--negative and continuous density $k$, such that $\mu(ds) = k(s)ds$, we see that $k\not\in L^1(0,\infty)$ because $M(t)\to\infty$ as $t\to\infty$. In particular the property \eqref{eq.fademu} is implied by $k(t)\to 0$ as $t\to\infty$. Therefore, it is perfectly possible for $k$ to lie in another $L^p$ space, for some $p>1$. As an example, suppose that $k(t)\sim t^{-\theta}$ as $t\to\infty$ for $\theta\in (0,1)$: then for $p>1/\theta>1$, $k\in L^p(0,\infty)$, while $k\not\in L^1(0,\infty)$. In this sense, our work shares concerns with existing results in the literature in which the Volterra equation does not possess an integrable kernel (see e.g., \cite{shea1975variants,hannsgen}).

The type of fading memory property \eqref{eq.fademu} we suggested was of interest motivates a stronger assumption on $M$. First, we see that 
\eqref{eq.fademu} implies  
\[
\frac{1}{n\tau} M(n\tau) = \frac{1}{\tau} \frac{1}{n}\sum_{j=0}^{n-1} \int_{[j\tau,j\tau+\tau)} \mu(ds)
\to 0 \text{ as $n\to\infty$}
\] 
and so the non--negativity of $\mu$ implies that $M(t)/t \to 0$ as $t\to\infty$. Since $M(t)\to\infty$ as $t\to\infty$, $M$ is non--decreasing, and $M(t)/t\to 0$ as $t\to\infty$, it is reasonable to suppose that $M \in \text{RV}_\infty(\theta)$ for $\theta\in[0,1]$. We note that the inclusion of $\theta=1$ in the parameter range does not lead to any problems in the analysis, and indeed it transpires that our arguments are valid for all $\theta\geq 0$. 

Analogously, the nonlinearity, $f$, is a positive and asymptotically increasing function such that $f(x)\to\infty$ and $f(x)/x\to 0$ as $x\to\infty$; hence it is natural to assume that $f\in \text{RV}_\infty(\beta)$ for $\beta\in [0,1)$. We can rule out some choices of the parameter $\beta$ rapidly: if $\beta>1$, $f(x)/x\to\infty$ as $x\to\infty$, and if $\beta<0$, $f$ is asymptotic to a decreasing function. When $\beta=0$ we append the hypotheses of asymptotic monotonicity and increase to infinity on $f$, as these are not necessarily satisfied by functions in $\text{RV}_\infty(0)$, but otherwise the analysis is essentially the same as when $\beta\in (0,1)$. The exclusion of the case $\beta=1$ is largely on technical grounds: informally, when $\beta=1$, the inverse of the increasing function $F$ defined by \eqref{cap_F} is \emph{no longer regularly varying}; $F^{-1}$ now belongs to the class of rapidly varying functions (which we define below). It also can be seen from the nature of our results that the asymptotic behaviour of solutions must be of a different form from those that hold when $\beta<1$. For $\beta<1$, no such technical problem arises, and indeed $F^{-1}$ is regularly varying with index $1/(1-\beta)$.      
 
In some situations, we will consider very rapidly growing forcing terms $H$ in the perturbed equation \eqref{volterra_pert} which are not regularly varying. We sometimes consider forcing terms from the class of \emph{rapidly varying functions}, and a definition of this class follows.
\begin{definition}
Suppose a measurable function $h:\mathbb{R}\to (0,\infty)$ obeys for $\lambda>0$:
\[
\lim_{x\to\infty}\frac{h(\lambda\,x)}{h(x)} = 
\begin{cases}
0, & \lambda<1, \\
1, & \lambda=1, \\
+\infty, & \lambda>1.
\end{cases}
\]
Then $h$ is \emph{rapidly varying at infinity}, or $h\in \text{RV}_\infty(\infty)$. If on the other hand, 
$h:\mathbb{R}\to (0,\infty)$ obeys for $\lambda>0$:
\[
\lim_{x\to\infty}\frac{h(\lambda\,x)}{h(x)} = 
\begin{cases}
+\infty, & \lambda<1, \\
1, & \lambda=1, \\
0, & \lambda>1.
\end{cases}
\]
Then we write $h\in \text{RV}_\infty(-\infty)$.
\end{definition}

The proof of our main result for \eqref{functional}, Theorem \ref{main_thm}, relies principally upon comparison methods, properties of regularly varying functions and a time change argument for delay differential equations. We first use constructive comparison methods, similar in spirit to those employed by Appleby and Buckwar \cite{applebycomparison2010} for linear equations, to establish ``crude'' upper and lower bounds on the solution of \eqref{functional}. The more challenging construction is that of the lower bound and is completed by comparing solutions of \eqref{functional} with those of a related nonlinear pantograph equation using time change arguments inspired by Brunner and Maset \cite{brunnermaset2009}. Finally, we prove a convolution lemma for regularly varying functions (cf. \cite[Theorem 3.4]{appleby2010regularly}) which is then used, in conjunction with straightforward comparison methods, to sharpen the aforementioned ``crude'' upper and lower bounds, and show that they coincide. Another paper which 
uses similar iterative methods to sharpen estimates in the growth of solutions of nonlinear convolution Volterra equations is Schneider~\cite{schneider1982}.

 With $\bar{M}(t):= \int_0^t M(s)ds$, we obtain $\lim_{t\to\infty}F(x(t))/\bar{M}(t) = \Lambda(\beta,\theta)$, or that the growth rate of solutions of \eqref{functional} depend explicitly on both indices of regular variation, and therefore the memory of the system (Theorem \ref{main_thm}). The value of the parameter--dependent limit $\Lambda$ can be determined explicitly in terms of the Gamma function. This result is only valid for $\beta \in [0,1)$ and hence may not hold if $f$ is only assumed to be sublinear (i.e. $\lim_{x\to\infty}f(x)/x=0$). In this sense, it appears that the imposition of the hypothesis of regular variation on $f$ and $M$ is intrinsic to the form of the asymptotic behaviour deduced, rather than a being a purely technical contrivance, and the restriction to $\beta\not=1$ also seems justified by grounds other than the complexity of the analysis needed to prove a sharp result. 

The results and methods outlined above for \eqref{functional} can also be used to yield sharp asymptotics for the perturbed equation \eqref{functional_pert}. 
If $H$ is positive, solutions to \eqref{functional_pert} will be positive and exhibit unbounded growth; therefore there is no need to assume pointwise positivity of $h$. Hence solutions of \eqref{functional_pert} are no longer necessarily non--decreasing and  more delicate comparison techniques are required to treat this additional difficulty.

When $H$ is of the same order of magnitude as the solution of \eqref{ODE}, we establish non-trivial upper and lower bounds on the solution and then employ a simple fixed point iteration argument to calculate the exact asymptotic growth rate of the solution in terms of a characteristic equation (Theorem \ref{thm.pert}). Moreover, the converse also holds: growth in the solution of \eqref{functional_pert} at a rate proportional to that of the solution of \eqref{ODE} is possible only when $H$ is of the same order as that solution. 
In these results, the parameter $\theta$ characterises the dependence of the growth rate on the degree of memory in the system. When the perturbation term grows sufficiently quickly, the solution tracks $H$ asymptotically, in the sense that $\lim_{t\to\infty}x(t)/H(t)=1$, even when $H$ is allowed to be highly non-monotone. Indeed, under certain restrictions we can show that our characterisation of rapid growth in the perturbation is necessary in order for $\lim_{t\to\infty}x(t)/H(t)=1$ to prevail.
\section{Main Results and Discussion}\label{results}
The following equivalence relation on the space of positive continuous functions and shorthand are used throughout.
\begin{definition}
Suppose $a,\,b \in C(\mathbb{R}^+;\mathbb{R}^+)$. $a$ and $b$ are asymptotically equivalent if 
$\lim_{t\to\infty}a(t)/b(t)=1$; we often write $a(t) \sim b(t)$ as $t\to\infty$ for short.
\end{definition}
$\mu$ is a non-negative Borel measure on $\mathbb{R}^+$ with infinite total variation; more precisely
\begin{align}\label{infinite_measure}
\mu(E) \geq 0 \text{ for all } E \in \mathcal{B}(\mathbb{R^+}), \quad \int_{[0,\infty)} \mu(ds) = \lim_{t\to\infty}M(t) = \infty,
\end{align}
where $M$ is defined as in \eqref{M_defn}. Our first result gives precise information on the asymptotic growth rate of the solution to \eqref{functional}. We state our result before carefully analysing the conclusion. We defer the proof to Section \ref{sec.main_proofs}.
\begin{theorem}\label{main_thm}
Suppose the measure $\mu$ obeys \eqref{infinite_measure} with $M \in \text{RV}_\infty(\theta),\, \theta \geq 0$ and that $f \in \text{RV}_\infty(\beta),\,\beta\in [0,1)$. When $\beta=0$ let $f$ be asymptotically increasing and obey $\lim_{x\to\infty}f(x)=\infty$. Then the solution, $x$, of \eqref{functional} satisfies $x \in \text{RV}_\infty\left((1+\theta)/(1-\beta)\right)$ and
\begin{align}\label{Lambda}
\lim_{t \to \infty}\frac{F(x(t))}{\bar{M}(t)} = \frac{\Gamma(\theta+1)\Gamma\left(\tfrac{1+\beta\theta}{1-\beta}\right)}{\Gamma\left(\tfrac{1+\theta}{1-\beta}\right)}=:\Lambda(\beta,\theta),
\end{align}
where $\Gamma(x) := \int_0^\infty t^{x-1}e^{-t}dt$ and $\bar{M}(t):= \int_0^t M(s)ds$.
\end{theorem}
By Karamata's Theorem (cf. \cite[Theorem 1.5.11]{BGT}), $\lim_{t\to\infty}\bar{M}(t)/tM(t) = 1/(1+\theta)$. Hence the conclusion of Theorem \ref{main_thm} is equivalent to 
\[
\lim_{t \to \infty}\frac{F(x(t))}{t\,M(t)} = (1+\theta)\frac{\Gamma(\theta+1)\Gamma\left(\tfrac{1+\beta\theta}{1-\beta}\right)}{\Gamma\left(\tfrac{1+\theta}{1-\beta}\right)} = \frac{1}{1-\beta}B\left(\theta +1,\frac{1+\theta \beta}{1-\beta}\right),
\]
where $B$ denotes the Beta function, which is defined by $B(x,y):=\int_0^1 \lambda^{x-1}(1-\lambda)^{y-1}d\lambda$ (cf. \cite[p.142]{nist}). Furthermore, since $F^{-1} \in \text{RV}_\infty\left(1/(1-\beta)\right)$, \eqref{Lambda} is also equivalent to
\begin{align}\label{x_asym_F_inv}
\lim_{t \to \infty}\frac{x(t)}{F^{-1}(t\,M(t))} = {\left\{\frac{1}{1-\beta}B\left(\theta +1,\frac{1+\theta\beta}{1-\beta}\right)\right\}}^{\tfrac{1}{1-\beta}}.
\end{align}
Theorem \ref{main_thm} expresses the leading order asymptotics of the solution in terms of the functions $F$ and $\bar{M}$. The dependence of $\Lambda$ on $\beta$ and $\theta$ is known explicitly and this can be used to gain some insight into second order effects of the nonlinearity and the memory on the growth rate. The following proposition records some properties of the function $\Lambda(\beta,\theta)$ that are useful when interpreting the conclusion of Theorem \ref{main_thm}. The proofs of the forthcoming claims are deferred to Section \ref{sec.examples}.
\begin{proposition}\label{interpret}
Suppose $\Lambda(\beta,\theta)$ is defined by \eqref{Lambda} with $\beta\in[0,1)$ and $\theta\in[0,\infty)$. Then
\begin{enumerate}[(i.)]
\item $\Lambda(0,\theta)=1$ for fixed $\theta\in(0,\infty)$ and $\Lambda(\beta,0)=1$ for fixed $\beta\in(0,1)$,
\item $\lim_{\beta\uparrow 1}\Lambda(\beta,\theta)=0$ for fixed $\theta\in(0,\infty)$ and $\lim_{\theta\to\infty}\Lambda(\beta,\theta)=0$ for fixed $\beta\in(0,1)$,
\item $\beta \mapsto \Lambda(\beta,\theta)$ is decreasing, $\beta\in(0,1)$, $\theta \mbox{ (fixed) } \in (0,\infty)$,
\item $\theta \mapsto \Lambda(\beta,\theta)$ is decreasing, $\theta\in(0,\infty)$, $\beta \mbox{ (fixed) } \in (0,1)$,
\item $\Lambda(\beta,\theta) \in (0,1)$ for $\beta\in(0,1)$ and $\theta\in(0,\infty)$.
\end{enumerate}
\end{proposition}
\begin{figure}[H]
\includegraphics[scale=0.43]{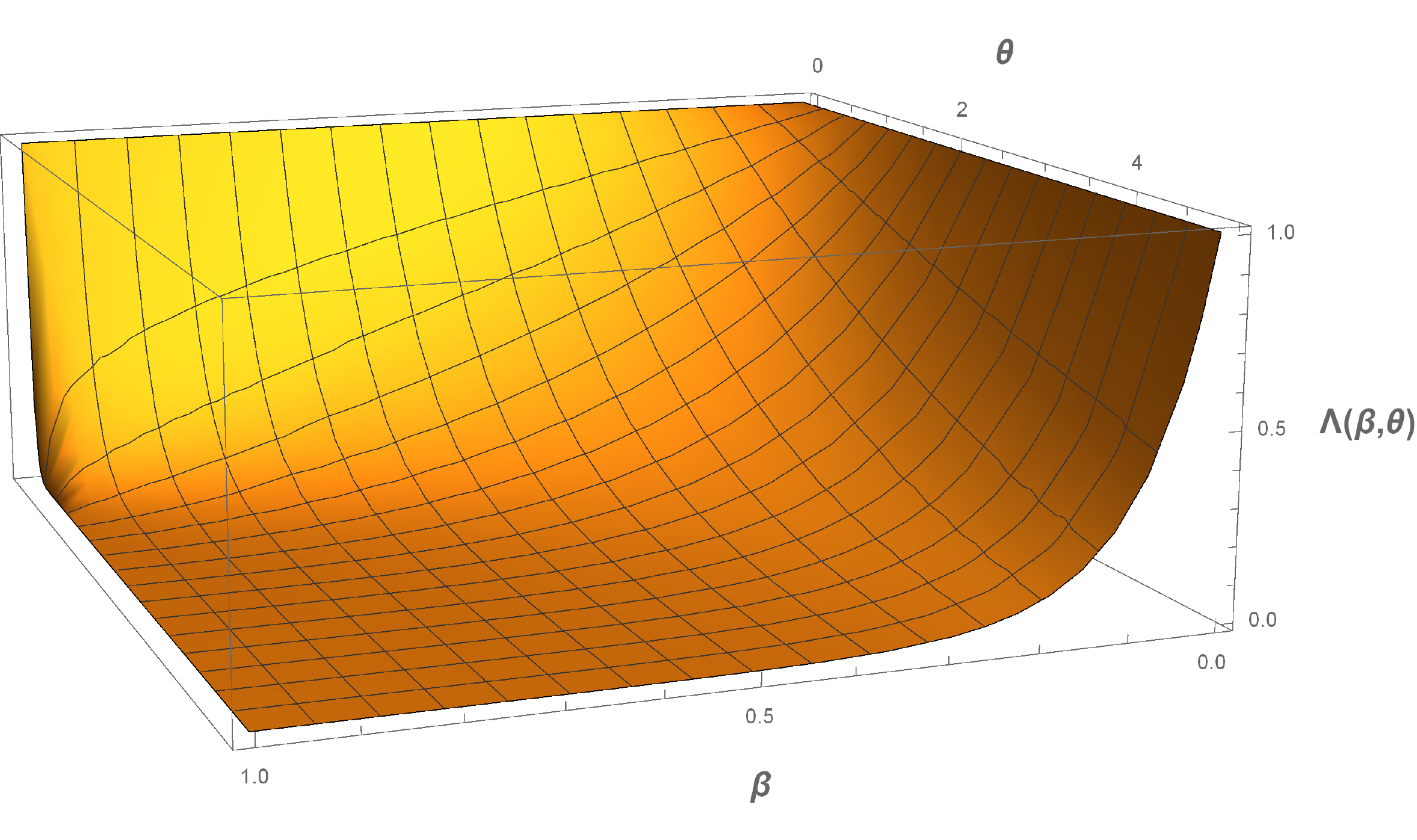}
\caption{Plot of the surface $\Lambda(\beta,\theta)$ with $\beta\in[0,1)$ and $\theta\in[0,5]$.}
\end{figure}
For each fixed $\beta \in (0,1)$, letting $\theta = 0$ in Theorem \ref{main_thm} yields $\lim_{t\to\infty}F(x(t))/\bar{M}(t)=1$. The authors have previously obtained this conclusion for sublinear equations of the form \eqref{volterra} without regular variation but with $\lim_{t\to\infty}M(t) = M\in(0,\infty)$. Therefore Theorem \ref{main_thm} can be thought of as a continuous extension of our previous results for \eqref{functional} with sublinear nonlinearities and finite measures (see \cite{sublinear2015} for further details).\\\\*
For a fixed $\beta \in (0,1)$, a decrease in the value of $\theta$ represents an increase in the rate of decay of the measure $\mu$. This can be made precise by supposing that the measure $\mu$ is absolutely continuous, and specifically that $\mu(ds) = m(s)ds$ for continuous $m\in \text{RV}_\infty(\theta-1),\,\theta\in(0,1)$.  Therefore, increasing the value of $\theta$ gives more weight to values of the solution in the past (more memory) and we expect the growth rate of solutions of \eqref{functional} to be slower than that of the related ordinary differential equation
\begin{align}\label{ODE}
y'(t) = M(t)f(y(t)),\,\, t >0; \,\, y(0) = \xi >0.
\end{align}
The equation \eqref{ODE}, in contrast, places the entire weight $M(t)$ at the present time, when the solution is largest. Hence, increasing the value of $\theta$ (putting more weight further into the past) slows the growth rate and it is intuitive that $\Lambda(\beta,\theta)$ is decreasing in $\theta$. Using this comparison with \eqref{ODE} once more, it is clear that Proposition \ref{interpret} $(v.)$ must hold since solutions of \eqref{functional} can never grow faster than those of \eqref{ODE} (if $f$ is strictly increasing this can be seen by inspection).

For a fixed $\theta \in (0,\infty)$, one might expect an increase in $\beta$ to lead to a faster rate of growth of the solution of \eqref{functional}. Therefore, it may initially be surprising  that $\Lambda(\beta,\theta)$ is decreasing in $\beta$. This counter-intuitive result is best understood by explaining the error introduced in the approximation of the right-hand side of \eqref{functional}. From \eqref{functional}
\[
x'(t) = \int_{[0,t]} \mu(ds)f(x(t-s)) = \int_{[0,t]} \mu(ds)\frac{f(x(t-s))}{f(x(t))} f(x(t)), \quad t >0.
\]  
The error of our upper bound on the solution is proportional to the ratio $f(x(t-s))/f(x(t))$ for $s\in(0,t)$, or $f(x(\lambda t))/f(x(t))$ for $\lambda \in (0,1)$. Since $f \circ x \in \text{RV}_\infty(\beta(1+\theta)/(1-\beta))$ 
\[
\lim_{t\to\infty}\frac{f(x(\lambda t))}{f(x(t))} = \lambda^{\frac{\beta(1+\theta)}{1-\beta}} =: \gamma(\beta). 
\]
When $\gamma(\beta)$ is close to one, the solution of \eqref{functional} is close to that of \eqref{ODE} and hence our estimate is sharp. However, $\gamma(\beta)$ is decreasing and $\lim_{\beta\uparrow 1}\gamma(\beta)=0$. Thus the zero limit as $\beta\uparrow 1$ in Proposition \ref{interpret} $(ii.)$ represents the fact that the solution of \eqref{ODE} increases much faster in $\beta$ than the solution to \eqref{functional}, for a fixed value of $\theta$.
\section{Results for Perturbed Volterra Equations}\label{perturbed}
We now present a result which illustrates how our precise understanding of the asymptotics of solutions of \eqref{functional} can be applied to perturbed versions of the equation, such as \eqref{functional_pert}. This result applies to perturbations of \eqref{functional} that are of the same, or smaller, order of magnitude as solutions of the ordinary differential equation \eqref{ODE}. Our assumptions on $H$ guarantee that $\lim_{t\to\infty}x(t) = \infty$ but this limit is no longer necessarily achieved monotonically and this is reflected in the added complexity of certain technical aspects of the proofs. The proofs of the results in this section are largely deferred to Section \ref{sec.main_proofs}.
\begin{theorem}\label{thm.pert}
Suppose the measure $\mu$ obeys \eqref{infinite_measure} with $M \in \text{RV}_\infty(\theta),\, \theta \geq 0$ and that $f \in \text{RV}_\infty(\beta),\,\beta\in [0,1)$. When $\beta=0$ let $f$ be asymptotically increasing and obey $\lim_{x\to\infty}f(x)=\infty$. Let $x$ denote the solution of \eqref{functional_pert} and suppose $H \in C((0,\infty);(0,\infty))$. Then the following are equivalent:
\begin{align}
(i.)\quad \lim_{t\to\infty}\frac{x(t)}{F^{-1}(t\,M(t))} = \zeta \in [L,\infty),\quad (ii.)\quad\lim_{t\to\infty}\frac{H(t)}{F^{-1}(t\,M(t))} = \lambda \in [0,\infty),
\end{align}
where $L = \left\{B\left(1+\theta,\tfrac{1+\theta\beta}{1-\beta}\right)/(1-\beta)\right\}^{1/(1-\beta)}$, and moreover
\begin{align}\label{characteristic}
\zeta = \frac{\zeta^\beta}{1-\beta}\,B\left(1+\theta,\tfrac{1+\theta\beta}{1-\beta}\right) + \lambda.
\end{align}
\end{theorem}
We notice that, when there is a sufficiently slowly growing forcing term $H$, $\lambda=0$, 
and we recover from \eqref{characteristic} exactly the asymptotic behaviour of the unperturbed equation, 
given by \eqref{x_asym_F_inv}. Also, in the limit as $\lambda\to 0^+$, the rate of the unperturbed equation is recovered. 

Condition $(ii.)$ on $H$ in Theorem~\ref{thm.pert} does not cover the case when $H$ is of larger magnitude than the solution of the unperturbed equation \eqref{functional} (or that of \eqref{ODE}). To deal with this case, we would like to know the growth rate of the solution when $\lim_{t\to\infty}H(t)/F^{-1}(t\,M(t))=\infty$. 
Insight into what happens can be gained by sending $\lambda\to\infty$ in Theorem~\ref{thm.pert}. For $\lambda>0$, from Theorem~\ref{thm.pert}, we have
\[
\lim_{t\to\infty} \frac{x(t)}{H(t)}=\frac{\zeta(\lambda)}{\lambda}=:\eta(\lambda),
\]
where $\zeta$ depends on $\lambda$ through \eqref{characteristic}. Since $\zeta=\zeta(\lambda)$ is the 
unique positive solution of \eqref{characteristic}, $\eta=\eta(\lambda)$ is the unique positive solution of 
$\eta=1+K\eta^\beta \lambda^{\beta-1}$ where $K>0$ is the $\lambda$--independent positive quantity
\[
K=\frac{1}{1-\beta}\,B\left(1+\theta,\tfrac{1+\theta\beta}{1-\beta}\right).
\]
Clearly $\eta(\lambda)>1$ and $\lambda\mapsto \eta(\lambda)$ is in $C^1$, by the implicit function theorem. Moreover, by implicit differentiation, $\eta'(\lambda)$ obeys 
\[
\eta'(\lambda)\left\{1-\beta\frac{\eta(\lambda)-1}{\eta(\lambda)}\right\}=K(\beta-1)\eta(\lambda)^\beta \lambda^{\beta-2}. 
\] 
Therefore, as the bracket on the left--hand side is positive, $\lambda\mapsto\eta(\lambda)$ is decreasing. Hence for $\lambda>1$, we have $\eta(\lambda)<1+K\eta(1)^\beta \lambda^{\beta-1}$, so $\limsup_{\lambda\to\infty} \eta(\lambda)\leq 1$, 
and so $\eta(\lambda)\to 1$ as $\lambda\to\infty$.

In view of this discussion, one might expect that $\lim_{t\to\infty}H(t)/F^{-1}(t\,M(t))=\infty$ implies 
$x(t)\sim H(t)$ as $t\to\infty$, or less precisely that sufficiently rapid growth in $H$ forces $x(t)$ to grow at the rate $H(t)$. Therefore, it is natural to ask under what conditions we would have $x(t)\sim H(t)$ as $t\to\infty$. It is straightforward to show that a necessary condition for $\lim_{t\to\infty}x(t)/H(t)=1$ is that $\lim_{t\to\infty}\int_0^t M(t-s)f(H(s))ds/H(t) = 0$. This motivates the hypothesis
\begin{align}\label{bigpert_hypth}
\lim_{t\to\infty}\frac{M(t)\int_0^t f(H(s))\,ds}{H(t)}=0,
\end{align}
and the following result. This result requires no monotonicity in $H$ and as such allows for $H$ to undergo considerable fluctuation, a point we will illustrate further in Section \ref{examples}.
\begin{theorem}\label{thm.bigpert}
Suppose the measure $\mu$ obeys \eqref{infinite_measure} with $M \in \text{RV}_\infty(\theta),\, \theta \geq 0$ and that $f \in \text{RV}_\infty(\beta),\,\beta\in [0,1)$. When $\beta=0$ let $f$ be asymptotically increasing and obey $\lim_{x\to\infty}f(x)=\infty$. Let $H$ be a function in $C((0,\infty);(0,\infty))$ satisfying \eqref{bigpert_hypth}. Then the solution, $x$, of \eqref{functional} obeys $\lim_{t\to\infty}x(t)/H(t) = 1$.
\end{theorem}
When $H$ is regularly varying at infinity the hypotheses of Theorems \ref{thm.pert} and \ref{thm.bigpert} align to give a complete classification of the asymptotics (Corollary \ref{equivalences}). However, assuming regular variation of $H$ imposes considerable regularity constraints. In particular, $H$ is then asymptotic to an increasing function and this restricts potential applications of Theorem \ref{thm.bigpert} to stochastic functional differential equations.
\begin{corollary}\label{equivalences}
Let $M \in \text{RV}_\infty(\theta),\, \theta \geq 0$ with $\lim_{t\to\infty}M(t)=\infty$. Suppose that $f \in \text{RV}_\infty(\beta),\,\beta\in [0,1)$. When $\beta=0$ let $f$ be asymptotically increasing and obey $\lim_{x\to\infty}f(x)=\infty$. If $H \in \text{RV}_\infty(\alpha), \alpha>0$, then the following are equivalent:
\begin{enumerate}[(i.)]
\item $\lim_{t\to\infty}M(t)\int_0^t f(H(s))\,ds/H(t) = 0$,
\item $\lim_{t\to\infty}H(t)/F^{-1}(tM(t))=\infty$,
\item $\lim_{t\to\infty}\int_0^t M(t-s) f(H(s))\,ds/H(t) = 0$.
\end{enumerate}
\end{corollary}
We exclude the case $\alpha=0$, because it is covered by Theorem~\ref{thm.pert} with $\lambda=0$.
\begin{proof}[Proof of Corollary \ref{equivalences}]
By hypothesis $f \circ H \in \text{RV}_\infty(\alpha\beta)$ and $M \in \text{RV}_\infty(\theta)$. Hence $\int_0^t f(H(s))\,ds \in \text{RV}_\infty(1+\alpha\beta)$ and Karamata's Theorem yields $\int_0^t f(H(s))ds \sim t\,f(H(t))/(1+\alpha\beta)$, as $t\to\infty$. Thus
\begin{align}\label{asym2}
\frac{M(t)\int_0^t f(H(s))\,ds}{H(t)} \sim \frac{M(t)\,t\,f(H(t))}{(1+\alpha\beta)H(t)}, \text{ as } t\to\infty.
\end{align}
Lemma \ref{Beta_Lemma} yields
\begin{align}\label{asym1}
\int_0^t M(t-s)f(H(s))\,ds \sim B(1+\alpha\beta,1+\theta)\, t\,M(t)\,f(H(t)),\text{ as } t \to\infty.
\end{align}
Therefore \eqref{asym2} and \eqref{asym1} together yield 
\[
\frac{\int_0^t M(t-s)f(H(s))\,ds}{H(t)}  \sim (1+\alpha\beta)B(1+\alpha\beta,1+\theta)\frac{M(t) \int_0^t f(H(s))\,ds}{H(t)}, \text{ as } t \to\infty.
\] 
Hence $(i.)$ and $(iii.)$ are equivalent. By Karamata's Theorem, $F(H(t)) \sim H(t)/((1-\beta)f(H(t)))$ or $f(H(t))/H(t) \sim 1/(1-\beta)F(H(t))$, as $t\to\infty$. Hence \eqref{asym2} can be restated as
\[
\frac{M(t)\int_0^t f(H(s))\,ds}{H(t)} \sim \frac{M(t)\,t}{(1+\alpha\beta)(1-\beta)F(H(t))}, \text{ as } t \to\infty.
\]
Thus if $(i.)$ holds, then $\lim_{t\to\infty}M(t)t/F(H(t))=0$. This implies $\lim_{t\to\infty}F(H(t))/M(t)t=\infty$ and hence that $(iii.)$ holds, by the regular variation of $F^{-1}$. The reverse implications are all also true and $(i.)$ and $(ii.)$ are equivalent.
\end{proof}
We state without proof a partial converse to Theorem \ref{thm.bigpert} with $H\in \text{RV}_\infty(\alpha),\,\alpha>0$. The proof follows from Corollary \ref{equivalences} and estimation arguments similar to those used throughout this paper. 
\begin{theorem}
Suppose the measure $\mu$ obeys \eqref{infinite_measure} with $M \in \text{RV}_\infty(\theta),\, \theta \geq 0$ and that $f \in \text{RV}_\infty(\beta),\,\beta\in [0,1)$. When $\beta=0$ let $f$ be asymptotically increasing and obey $\lim_{x\to\infty}f(x)=\infty$. Let $x$ denote the solution of \eqref{functional_pert}, $H \in C((0,\infty);(0,\infty)) \cap \text{RV}_\infty(\alpha)$ with $\alpha>0$. Then the following are equivalent:
\[
(i.)\quad
\lim_{t\to\infty}\frac{M(t)\int_0^t f(H(s))\,ds}{H(t)}=0, \quad (ii.)\quad
\lim_{t\to\infty}\frac{x(t)}{H(t)}=1. 
\]
\end{theorem} 
While discussing the hypothesis that $\lim_{t\to\infty}H(t)/F^{-1}(tM(t))=\infty$ in the context of regular variation it is worth remarking that this hypothesis is also satisfied for $H \in \text{RV}_\infty(\infty)$, the so-called rapidly varying functions (see \cite[p.83]{BGT}). If $H \in \text{RV}_\infty(\infty)$, then \eqref{bigpert_hypth} holds and Theorem \ref{thm.bigpert} can be applied; this fact is recorded in the following corollary.
\begin{corollary}\label{rapid}
Suppose the measure $\mu$ obeys \eqref{infinite_measure} with $M \in \text{RV}_\infty(\theta),\, \theta \geq 0$ and that $f \in \text{RV}_\infty(\beta),\,\beta\in [0,1)$. When $\beta=0$ let $f$ be asymptotically increasing and obey $\lim_{x\to\infty}f(x)=\infty$. Let $x(t)$ denote the solution of \eqref{functional_pert} and suppose $H \in C((0,\infty);(0,\infty)) \cap \text{RV}_\infty(\infty)$ is asymptotically increasing. Then $\lim_{t\to\infty}x(t)/H(t) = 1$.
\end{corollary} 
\begin{proof}[Proof of Corollary \ref{rapid}] Define the function $a$ by
\[
\frac{M(t)\,t\,f(H(t))}{H(t)} =: M(t)\,t\,a(H(t)).
\]
Then $a\in \text{RV}_\infty(\beta-1)$ and $a\circ H \in \text{RV}_\infty(-\infty)$. Hence $\lim_{t\to\infty}M(t)\,t\,a(t)=0$ and because $f \circ H$ is asymptotically increasing \eqref{bigpert_hypth} holds. Applying Theorem \ref{thm.bigpert} then yields the desired conclusion.
\end{proof}
Corollary \ref{rapid} will also hold if $H \in \text{MR}_\infty(\infty)$, a sub-class of $\text{RV}_\infty(\infty)$ 
(see \cite[p.68]{BGT} for the definition of $\text{MR}_\infty(\infty)$), because this guarantees that $H$ is asymptotic to an increasing function (see \cite[p.83]{BGT}).
\section{Examples}\label{examples}
The requisite proofs and justifications supporting the discussion in this section are deferred to Section \ref{sec.examples}.
\subsection{Application of Theorem \ref{main_thm}}
The main attraction of Theorem \ref{main_thm} is that it largely reduces the asymptotic analysis of solutions of \eqref{functional} to the computation, or asymptotic analysis, of the function $F^{-1}$. This is because under the appropriate hypotheses, Theorem \ref{main_thm} yields
\[
x(t) \sim F^{-1}(t\,M(t)){\left\{\frac{1}{1-\beta}B\left(\theta +1,\frac{\theta \beta+1}{1-\beta}\right)\right\}}^{\tfrac{1}{1-\beta}}, \text{ as }t\to\infty.
\]
In general, exact computation of $F^{-1}$ in closed form is not possible. The following result provides the asymptotics of $F^{-1}$ for a large class of $f \in \text{RV}_\infty(\beta)$ for $\beta \in[0,1)$ using some classic results from the theory of regular variation. It's principal appeal is that it can be applied by calculating the limit of a readily--computed function which can be found directly in terms of $f$, without need for 
integration.   
\begin{proposition}\label{asym_lemma}
Suppose $f \in \text{RV}_\infty(\beta), \, \beta\in[0,1)$ is continuous and that $\ell(x) := \left(f(x)/x^\beta\right)^{\tfrac{1}{1-\beta}}$ obeys
\begin{align}\label{debruijn}
\lim_{x\to\infty}\frac{\ell(x\,\ell(x))}{\ell(x)} = 1.
\end{align}
Then 
\[
F(x) \sim \frac{1}{1-\beta}\frac{x}{f(x)}, \quad F^{-1}(x) \sim (1-\beta)^{\frac{1}{1-\beta}}\,\ell\left(x^{\frac{1}{1-\beta}}\right)x^{\frac{1}{1-\beta}}, \text{ as }x\to\infty.
\]
\end{proposition}
The following examples illustrate the convenience of Proposition \ref{asym_lemma} in practice. 
\begin{example}\label{examp_1}
Suppose $f(x) \sim a\,x^\beta \, \log\log(x^\alpha)$ as $x\to\infty$ with $\beta\in[0,1)$, $a>0$ and $\alpha>0$. In this case
\[
\ell(x) \sim \left( a\,\log\log\left(x^\alpha\right) \right)^\frac{1}{1-\beta}, \text{ as }x\to\infty.
\]  
Hence
\begin{align*}
\bar{L}(x):=\frac{\ell(x\,\ell(x))}{\ell(x)}= \frac{\left( a\,\log\log\left(x^\alpha\,\ell^\alpha(x)\right) \right)^\frac{1}{1-\beta}}{\left( a\,\log\log\left(x^\alpha\right) \right)^\frac{1}{1-\beta}}
= \frac{\left( \log\log\left(x^\alpha\,a^{\frac{\alpha}{1-\beta}}\,\left\{ \log\log(x^\alpha) \right\}^\frac{\alpha}{1-\beta}\right) \right)^\frac{1}{1-\beta}}{\left( \log\log\left(x^\alpha\right) \right)^\frac{1}{1-\beta}}, \,\, x>0.
\end{align*}
Let $\log\log(x^\alpha)=u$ to obtain
\begin{align*}
\bar{L}(\exp\exp(u)^{1/\alpha})&= \left(\frac{\log\log(x^\alpha\,a^\frac{\alpha}{1-\beta}\,u^\frac{\alpha}{1-\beta})}{u} \right)^\frac{1}{1-\beta} = \left(\frac{\log(e^u+\log(a^\frac{\alpha}{1-\beta})+\log(u^\frac{\alpha}{1-\beta}))}{u} \right)^\frac{1}{1-\beta}\\
&= \left(\log(e^u)/u+ \log\left(1+ \frac{\log(a^\frac{\alpha}{1-\beta})+\log(u^\frac{\alpha}{1-\beta})}{e^u}\right)/u \right)^\frac{1}{1-\beta} =: \left(1 + G(u)\right)^\frac{1}{1-\beta},
\end{align*}
where $\lim_{u\to\infty}G(u)=0$. Therefore $\lim_{x\to\infty}\bar{L}(x):=\ell(x\,\ell(x))/\ell(x)=1$ and applying Lemma \ref{asym_lemma} yields
\[
F^{-1}(x) \sim (1-\beta)^\frac{1}{1-\beta}\left\{ a\,\log\log\left(x^\frac{\alpha}{1-\beta}\right) \right\}^\frac{1}{1-\beta}x^\frac{1}{1-\beta}, \text{ as }x\to\infty.
\]
This example is also valid with $\log\log(x)$ replaced by $\prod_{i=1}^n \log_{i-1}(x)$, where $\log_i(x)=\log\log_{i-1}(x)$. The proof in this case is essentially the same but the resulting formulae are rather convoluted. 
\end{example}
\begin{example}\label{examp_2}
Suppose $f(x) \sim x^\beta \left(2+ \sin(\log\log(x))\right)$ as $x\to\infty$, with $\beta\in(0,1)$. In this case
\[
\ell(x) \sim \left(2+ \sin(\log\log(x))\right)^\frac{1}{1-\beta}, \text{ as }x\to\infty.
\]  
Hence
\[
\bar{L}(x) := \left(\frac{2+ \sin(\log\log(x\,\ell(x)))}{2+ \sin(\log\log(x))}\right)^{\frac{1}{1-\beta}}.
\]
Once more make the substitution $\log\log(x) = u$ to obtain
\begin{align*}
\bar{L}(\exp\exp(u))&= \left(\frac{2+ \sin(\log\log(\exp\exp(u)\{2+\sin(u)\}^{\tfrac{1}{1-\beta}}))}{2+ \sin(u)}\right)^{\frac{1}{1-\beta}} \\
&= \left(\frac{2+ \sin(\log(\exp(u)+ \log\{2+\sin(u)\}^{\tfrac{1}{1-\beta}}))}{2+ \sin(u)}\right)^{\frac{1}{1-\beta}}\\
&= \left(\frac{2+ \sin\left(u+ \log[1+\log\{2+\sin(u)\}^{\tfrac{1}{1-\beta}}/{e^u}]\right)}{2+ \sin(u)}\right)^{\frac{1}{1-\beta}}.
\end{align*}
Letting $u \to \infty$ in the above yields $\lim_{u\to\infty}\bar{L}(\exp\exp(u)) = 1$ and hence Proposition \ref{asym_lemma} applies. Therefore
\[
F^{-1}(x) \sim (1-\beta)^{\frac{1}{1-\beta}}\,\left\{2 + \sin\left(\log\log(x^{\frac{1}{1-\beta}})\right)\right\}^{\frac{1}{1-\beta}}x^{\frac{1}{1-\beta}}, \text{ as }x\to\infty.
\]
\end{example}
\subsection{Discrete Measures}
It may appear that our inclusion of a general measure $\mu$ in \eqref{functional} and the hypothesis that the integral of $\mu$ is regularly varying are only compatible when $\mu$ is an absolutely continuous measure. The following proposition allows us to easily construct examples to show our results also cover a variety of equations involving discrete measures.
\begin{proposition}\label{discrete}
Let $x\geq 0$ and $\delta_x$ be the Dirac measure at $x$ on $(\mathbb{R}^+,\mathcal{B}(\mathbb{R}^+)$.
Suppose that $\theta\in(0,1)$ and that $\mu_0\in \text{RV}_\infty(\theta-1)$. Let $\tau >0$ and 
\begin{align}\label{discrete_meas}
\mu(ds) = \sum_{j=0}^{\lfloor t/\tau \rfloor}\mu_0(j\tau)\delta_{j\tau}(ds).
\end{align} 
Hence 
\begin{align}\label{M_discrete}
M(t) = \int_{[0,t]} \mu(ds) = \sum_{j=0}^{\lfloor t/\tau \rfloor}\mu_0(j\tau),
\end{align}
and $M \in \text{RV}_\infty(\theta)$. Furthermore,
$
M(t) \sim \tilde{M}(t) := \int_0^t \tilde{\mu}(s)ds \text{ as }t\to\infty,
$
where $\tilde{\mu} \in \text{RV}_\infty(\theta-1)$ is any $C^1$, decreasing function such that $\mu_0(s) \sim \tilde{\mu}(s)$ as $s \to \infty$.
\end{proposition}
The following examples illustrate, using Proposition \ref{discrete}, the application of some of our results to equations involving discrete measures.
\begin{example}
Using the notation of Proposition \ref{discrete}, suppose that 
\begin{align}
x'(t) = \sum_{j=0}^{\lfloor t/\tau \rfloor}\mu_0(j\tau)f(x(t-j\tau))+ \int_0^t \mu_1(s)f(x(t-s))ds, \, t >0,
\end{align}
where $m$ given by $m(E)=\int_E \mu_1(s)ds$ for any Borel set $E\subset [0,\infty)$ is an absolutely continuous measure. Therefore
\begin{align}\label{meas_examp}
\mu(ds) = \sum_{j=0}^\infty \mu_0(j\tau)\delta_{j\tau}(ds) + \mu_1(s)ds.
\end{align}
If $\mu_0 \in \text{RV}_\infty(\theta-1)$ and $\mu_1 \in \text{RV}_\infty(\alpha)$, then $M \in \text{RV}_\infty(\max(\theta,\alpha+1))$. Suppose that $\theta > \alpha+1$ for the purposes of this example. Thus $M(t) \sim \sum_{j=0}^{\lfloor t/\tau \rfloor}\mu_0(j\tau)$ as $t\to\infty$ and choose \[
\mu_0(x) \sim \log(x+1)/(1+x)^{1-\theta} =: \tilde{\mu}(x),\quad \theta \in (0,1),
\]
where the asymptotic relation holds as $x\to\infty$. Hence $\tilde{\mu} \in \text{RV}_\infty(\theta-1)$ and it is straightforward to show that 
\[
\tilde{M}(t) = \frac{1}{\theta}(t+1)^{\theta}\log(t+1) - \frac{1}{\theta^2}((t+1)^{\theta}-1) \sim \frac{t^{\theta}}{\theta}\log(t), \text{ as }t\to\infty.
\]
Combining these facts, and Proposition \ref{discrete}, with Examples \ref{examp_1} and \ref{examp_2} we can provide the exact asymptotics for particular classes of solutions to \eqref{functional}. \\\\*
From Example \ref{examp_1} and \eqref{x_asym_F_inv}, when $\mu(ds)$ is given by \eqref{meas_examp} and $f(x) \sim a\,x^\beta \, \log\log(x^\alpha)$ as $x\to\infty$,
\begin{align*}
x (t) &\sim  \left\{ \frac{a\,t^{\theta+1}\log(t)}{\theta}\,\log\log\left(\frac{t^{\frac{\alpha(1+\theta)}{1-\beta}}\log(t)^{\frac{\alpha}{1-\beta}}}{\theta^{\frac{\alpha}{1-\beta}}} \right) \,B\left(\theta +1,\frac{1+\theta\beta}{1-\beta}\right)\right\}^{\tfrac{1}{1-\beta}}\\
&\sim \left\{ \frac{a\,t^{\theta+1}\log(t)}{\theta}\,\log\log\left(t^{\frac{\alpha(1+\theta)}{1-\beta}} \right) \,B\left(\theta +1,\frac{1+\theta\beta}{1-\beta}\right)\right\}^{\tfrac{1}{1-\beta}}, \text{ as }t\to\infty.
\end{align*}
Similarly, using Example \ref{examp_2}, with $f(x) \sim x^\beta \left(2+ \sin(\log\log(x))\right)$ as $x\to\infty$,
\begin{align*}
x(t) &\sim \left\{ \frac{t^{\theta+1}\log(t)}{\theta}\left[ 2+\sin\left(\log\log\left(\frac{t^{\frac{1+\theta}{1-\beta}}\log(t)^{\frac{1}{1-\beta}}}{\theta^{\frac{1}{1-\beta}}}\right)\right) \right]B\left(\theta +1,\frac{1+\theta\beta}{1-\beta}\right) \right\}^{\tfrac{1}{1-\beta}} \\
&\sim \left\{ \frac{t^{\theta+1}\log(t)}{\theta}\left[ 2+\sin\left(\log\log\left(t^{\frac{1+\theta}{1-\beta}}\right)\right) \right]B\left(\theta +1,\frac{1+\theta\beta}{1-\beta}\right) \right\}^{\tfrac{1}{1-\beta}}, \text{ as }t\to\infty.
\end{align*}
\end{example}
\subsection{Perturbed Equations and Application of Theorems \ref{thm.pert} and \ref{thm.bigpert}}
Using a parametrized example we illustrate how the asymptotic behaviour of solutions of \eqref{functional_pert} can be classified using the range of possibilities covered by the results in Section \ref{perturbed}. 
\begin{example}
For ease of exposition suppose that $\beta \in(0,1)$ and let
\[
f(x) = x^\beta, \quad x\geq 0; \quad M(t) = (1+t)^{\theta}-1, \quad t\geq 0;\quad 
H(t) = (1+t)^\alpha\,e^{\gamma t}-1,\quad t \geq 0,
\]
with $\theta>0$, $\alpha \in \mathbb{R}$, and $\gamma \geq 0$.
Hence \eqref{functional_pert} becomes
\[
x(t) = x(0)+ \int_0^t ((1+t-s)^{\theta}-1) x(s)^\beta ds +  (1+t)^\alpha\,e^{\gamma t}-1, \quad t \geq 0,
\]
with $x(0)>0$. Therefore
\[
F^{-1}(t\,M(t)) \sim (1-\beta)^{\tfrac{1}{1-\beta}}\,t^{\tfrac{\theta+1}{1-\beta}}, \text{ as }t\to\infty.
\]
\textbf{Case $(i.):\, \gamma=0.$} In this case $H\in \text{RV}_\infty(\alpha)$ and 
\[
\frac{H(t)}{F^{-1}(t\,M(t))} \sim (1-\beta)^{\tfrac{1}{\beta-1}} t^{\alpha-\frac{\theta+1}{1-\beta}}, \text{ as }t\to\infty.
\]
If $\alpha < (\theta+1)/(1-\beta)$, then $\lim_{t\to\infty}H(t)/F^{-1}(t\,M(t))=0$ and Theorem \ref{thm.pert} yields the limit 
\[
\lim_{t\to\infty}x(t)/F^{-1}(t\,M(t)) = L,
\]
 where
$
L = \left\{B\left(1+\theta,\tfrac{1+\theta\beta}{1-\beta}\right)/(1-\beta)\right\}^{1/(1-\beta)}.
$ 

If $\alpha = (\theta+1)/(1-\beta)$, then $\lim_{t\to\infty}H(t)/F^{-1}(t\,M(t))= (1-\beta)^{1/(\beta-1)}=: \lambda$ and Theorem \ref{thm.pert} gives $\lim_{t\to\infty}x(t)/F^{-1}(t\,M(t)) = \zeta$, where $\zeta$ satisfies \eqref{characteristic}.

Finally, if $\alpha > (\theta+1)/(1-\beta)$, then $\lim_{t\to\infty}H(t)/F^{-1}(t\,M(t))= \infty$. Then, by Corollary \ref{equivalences}, \eqref{bigpert_hypth} holds and Theorem \ref{thm.bigpert} yields $\lim_{t\to\infty}x(t)/H(t)=1$.

\textbf{Case $(ii.):\, \gamma>0.$} In this case $H \in \text{RV}_\infty(\infty)$ and Corollary \ref{rapid} immediately gives $\lim_{t\to\infty}x(t)/H(t)=1$ for all $\alpha \in \mathbb{R},\, \beta \in (0,1)$ and $\theta >0$. 
\end{example}
Particularly with a view to applications to stochastic functional differential equations, it is pertinent to highlight when $H$ is required to have some form monotonicity in the results of Section \ref{perturbed}. When $\lambda = 0$ in Theorem \ref{thm.pert} there is no monotonicity requirement on $H$ but $\lambda > 0$ implies that $H$ asymptotic to the monotone increasing function $F^{-1}$, modulo a constant. By contrast, Theorem \ref{thm.bigpert} allows for large ``fluctuations'', or irregular behaviour, in $H$; the following examples illustrate this point.
\begin{example}
Suppose $f \in \text{RV}_\infty(\beta), \, \beta \in (0,1)$, $M \in \text{RV}_\infty(\theta), \, \theta \geq 0$ and 
$H(t)= (1+t)^\alpha\, \left(2+\sin(t) \right)-2$, $\alpha>0$. From Karamata's Theorem 
\begin{align*}
\limsup_{t \to\infty}\frac{M(t) \int_0^t f(H(s))\,ds}{H(t)} 
&= \limsup_{t\to\infty}\frac{M(t) \int_0^t f((1+s)^\alpha\, \left(2+\sin(s)-2 \right))\,ds}{(1+t)^\alpha\, \left(2+\sin(t) \right)-2}\\ &\leq \limsup_{t\to\infty}\frac{(1+\epsilon)M(t) \int_0^t \phi(3\,s^\alpha)\,ds}{t^\alpha}.
\end{align*}
Since 
\[
\frac{M(t) \int_0^t \phi(3\,s^\alpha)\,ds}{t^\alpha} 
\sim \frac{M(t)\,t\,f(3\,t^\alpha)}{(1+\alpha\beta)t^\alpha}, \text{ as }t\to\infty,
\]
a sufficient condition for \eqref{bigpert_hypth} to hold, and hence for Theorem \ref{thm.bigpert} to apply, is $\alpha > (1+\theta)/(1-\beta)$. Even more rapid variation in $H$ is permitted; for example let $H(t) = e^t(2+\sin(t))-2$. In this case asymptotic monotonicity of $f$ and the rapid variation of $e^t$ yield
\[
\limsup_{t\to\infty}\frac{M(t) \int_0^t f(H(s))\,ds}{H(t)} \leq \limsup_{t\to\infty}\frac{M(t)\,t\,f(3e^t)}{e^t} =0,
\]
and once more Theorem \ref{thm.bigpert} applies to yield $x(t) \sim H(t)$ as $t\to\infty$, where $x$ is the solution to \eqref{volterra_pert}. By fixing $f(x) = x^\beta$, we can immediately see that it is possible to capture more general types of exponentially fast oscillation in Theorem \ref{thm.bigpert}. Choose $H(t) = e^{\sigma(t)\,t}$, where $\sigma(t)$ obeys $0 < \sigma_- \leq \sigma(t) \leq \sigma_+ < \infty$ for all $t \geq 0$, for some constants $\sigma_-$ and $\sigma_+$.
Checking condition \eqref{bigpert_hypth} yields
\[
\limsup_{t\to\infty}\frac{M(t) \int_0^t f(H(s))\,ds}{H(t)} 
\leq \limsup_{t\to\infty}\frac{M(t)\,t\,e^{\beta\sigma_+ t}}{e^{\sigma_- t}}.
\]
The limit of the right hand side will be zero if $\sigma_- > \beta\sigma_+$. 
\end{example}
Finally we present an example of a $H$ for which condition \eqref{bigpert_hypth} fails to hold. This example illustrates the limitations of our results by showing that when the exogenous perturbation exhibits rapid, irregular growth we are unable to capture the dynamics of the solution. This example is constructed by considering an extremely ill-behaved perturbation with periodic fluctuations of exponential order.
\begin{example}
Choose $f(x)=x^\beta$ and $H(t) \sim e^{t(1+\alpha p(t))} := H^*(t)$, as $t\to\infty$, with $\alpha \in (0,1)$, $\beta \in (0,1)$ and $p$ a continuous $1-$periodic function such that $\max_{t \in [0,1]}p(t)=1$ and $\min_{t \in [0,1]}p(t)=-1$. Let $t>0$ and define $n(t) \in \mathbb{N}$ such that $n(t) \leq t < n(t)+1$. Then 
\begin{align*}
S(t) &:= \int_0^t f(H^*(s))ds = \int_0^t e^{\beta s(1+\alpha p(s))}ds = \sum_{j=0}^{n(t)-1}\int_j^{j+1}e^{\beta s(1+\alpha p(s))}ds + \int_{n(t)}^te^{\beta s(1+\alpha p(s))}ds \\
&= \sum_{j=0}^{n(t)-1}\int_0^1 e^{\beta(u+j)(1+\alpha p(u))}du + \int_0^{t-n(t)}e^{\beta(u+n(t))(1+\alpha p(u))}du.
\end{align*}
Let $I_j := \int_0^1 e^{\beta(u+j)(1+\alpha p(u))}du$ and $S_n := \sum_{j=0}^{n-1} I_j$. Then 
$
S(t) = S_{n(t)} = \int_0^{t-n(t)}e^{\beta(u+n(t))(1+\alpha p(u))}du.
$
Hence $S(t) \leq S_{n(t)} = \int_0^{1}e^{\beta(u+n(t))(1+\alpha p(u))}du  = S_{n(t)+1}.$ Thus
$
S_{n(t)} \leq S(t) \leq S_{n(t)+1}.
$
\\\\* Now estimate $I_j$ as $j \to \infty$ as follows. Letting $c(u) = e^{u(1+\alpha p(u))\beta}$ and 
$d(u) = \beta(1+\alpha p(u))$ we have $I_j = \int_0^1 c(u) e^{jd(u)}du$. 
By hypothesis $0 < \underline{c} = \min_{u \in [0,1]}c(u) \leq \max_{u \in [0,1]}c(u) \leq \overline{c} < \infty$. 
Hence
\[
\underline{c}\int_0^1 e^{jd(u)}du \leq I_j \leq \overline{c}\int_0^1 e^{jd(u)}du, \,\, j \geq 0.
\]
Therefore, for $j \geq 1$,
\[
\underline{c}^{\tfrac{1}{j}}\left(\int_0^1 e^{jd(u)}\,du\right)^{\tfrac{1}{j}} \leq I_j^{\tfrac{1}{j}} \leq \overline{c}^{\tfrac{1}{j}} \left(\int_0^1 e^{jd(u)}\,du\right)^{\tfrac{1}{j}}.
\]
For any continuous, non-negative function $a: [0,1] \mapsto (0,\infty)$, $\lim_{j\to\infty}\left( \int_0^1 a(u)^j du \right)^{1/j} = \max_{u \in [0,1]}a(u)$ and thus $\lim_{j\to\infty}I_j^{1/j} = \max_{u \in [0,1]}e^{d(u)}$. Therefore
\[
\lim_{j\to\infty}\frac{1}{j}\log I_j = \max_{u \in [0,1]}d(u) = \beta \max_{u \in [0,1]}(1+\alpha p(u)) = \beta(1+\alpha) >0.
\]
Since $S_n = \sum_{j=0}^{n-1}I_j$, this gives us $\lim_{n\to\infty}\log S_n /n = \beta(1+\alpha)$. Hence
\[
\liminf_{t\to\infty}\frac{1}{t}\log S(t) \geq \liminf_{t\to\infty}\frac{1}{t}\log S_{n(t)} = \liminf_{t\to\infty}\frac{1}{n(t)}\log S_{n(t)}\frac{n(t)}{t} = \beta(1+\alpha).
\]
An analogous calculation for the limit superior then yields $\lim_{t\to\infty}\log S(t) /t = \lambda$. Therefore, 
as $\log H^*(t) /t = 1 + \alpha p(t)$,
\[
\limsup_{t\to\infty}\frac{1}{t}\log\left(\frac{H^*(t)}{S(t)}\right) = \limsup_{t\to\infty}\left(1+ \alpha p(t)\right) - \beta(1+\alpha) = (1+\alpha)(1-\beta) >0.
\]
Hence $\limsup_{t\to\infty}H^*(t) / \int_0^t f(H^*(s))\,ds = \infty$, and because $H(t)\sim H^\ast(t)$ as $t\to\infty$ and $f\in \text{RV}_\infty(\beta)$, we have $\limsup_{t\to\infty}H(t) / \int_0^t f(H(s))\,ds = \infty$. Similarly 
\[
\liminf_{t\to\infty}\frac{1}{t}\log \left(\frac{H^*(t)}{S(t)} \right) = \liminf_{t\to\infty}(1+\alpha p(t)) - \beta(1+\alpha) = 1-\beta - \alpha(1+\beta).
\]
Choose $\alpha > (1-\beta)/(1+\beta) >0$ to ensure that $1-\beta - \alpha(1+\beta) < 0$ and $\liminf_{t\to\infty}H^*(t) / \int_0^t f(H^*(s))\,ds = 0$. As above, this gives $\liminf_{t\to\infty}H(t) / \int_0^t f(H(s))\,ds = 0$. We remark that because the function $t\mapsto H(t) / \int_0^t f(H(s))\,ds$ is of exponential order, \eqref{bigpert_hypth} is violated for any $M \in \text{RV}_\infty(\theta),$ $\theta \in [0,\infty)$.
\end{example}
\section{Proofs of Results}\label{sec.main_proofs}
In the proofs that follow we will often choose to work with a monotone function approximating $f$. This monotone approximation will be denoted by $\phi$. If $f$ is regularly varying with a positive index then 
\begin{align}
&\text{There exists $\phi\in C^1((0,\infty);(0,\infty))\cap C(\mathbb{R}^+,(0,\infty))$ such that }\nonumber \\
\label{f_monotone_approx}
&f(x) \sim \phi(x) \text{ and $\phi'(x) > 0$ for all $x>0$}. 
\end{align}
by \cite[Theorem 1.3.1 and Theorem 1.5.13]{BGT}. It is immediate that if $f$ is regularly varying and asymptotic to $\phi$, then $\phi$ is also regularly varying with the same index. If $f \in \text{RV}_\infty(0)$ we assume that a $\phi$ satisfying \eqref{f_monotone_approx} exists since only a smooth, but not necessarily monotone, approximation is guaranteed in this case. The function $F(x)$ is approximated by $\Phi(x) := \int_1^x du/\phi(u)$ and $\Phi^{-1}$ is the inverse function of $\Phi$. If $f(x)\sim\phi(x)$ as $x\to\infty$ it follows trivially that $F(x) \sim \Phi(x)$ and $F^{-1}(x) \sim \Phi^{-1}(x)$, as $x\to\infty$.

The proof of Theorem \ref{main_thm} is decomposed into the following lemmata, the first of which provides a precise estimate on the asymptotics of the convolution of two regularly varying functions. 
\begin{lemma}\label{Beta_Lemma}
Suppose $a \in \text{RV}_\infty(\rho)$ and $b \in \text{RV}_\infty(\sigma)$, where $\rho \geq 0$ and $\sigma \geq 0$, and $\lim_{t\to\infty}a(t)=\infty$. If $\sigma = 0$ let $b$ be asymptotically increasing and obey $\lim_{t\to\infty}b(t) = \infty$. Then
\[
\lim_{t\to\infty}\frac{\int_0^t a(s)b(t-s)ds}{t\,a(t)\,b(t)} = \int_{0}^{1} \lambda^{\rho}(1-\lambda)^{\sigma}d \lambda =: B(\rho +1,\sigma + 1),
\]
where $B$ denotes the Beta function.
\end{lemma}
\begin{proof}
Let $\epsilon, \, \eta \in (0,\tfrac{1}{2})$ be arbitrary. Define
\begin{align}\label{int_a_b}
I(t) &:= \int_0^t a(s)b(t-s)\,ds = \int_0^{\epsilon t}a(s)\,b(t-s)\,ds + \int_{\epsilon t}^{(1-\eta)t}a(s)b(t-s)\,ds
+ \int_{(1-\eta)t}^t a(s)b(t-s)\,ds \nonumber\\
&=: I_1(t) + I_2(t) + I_3(t).
\end{align}
By making the substitution $s = \lambda t$
\[
\frac{I_2(t)}{t\,a(t)\,b(t)} = \frac{\int_{\epsilon t}^{(1-\eta)t}a(s)b(t-s)\,ds}{t \,a(t)\, b(t)} = \int_{\epsilon}^{1-\eta}\frac{a(\lambda t)}{a(t)} \frac{b(t(1-\lambda))}{b(t)} d\lambda.
\]
By the Uniform Convergence Theorem for Regularly Varying Functions (see \cite[Theorem 1.5.2]{BGT}) it follows that
\begin{align}\label{lim_I_2}
\lim_{t\to\infty}\frac{I_2(t)}{t\,a(t)\,b(t)} = \int_{\epsilon}^{1-\eta} \lambda^{\rho}(1-\lambda)^{\sigma}d \lambda.
\end{align}
Since both $a$ and $b$ are positive functions it is clear that $I(t) \geq I_2(t)$ and hence
\[
\liminf_{t\to\infty}\frac{I(t)}{t\,a(t)\,b(t)} \geq  \int_{\epsilon}^{1-\eta} \lambda^{\rho}(1-\lambda)^{\sigma}d \lambda.
\]
Letting $\eta$ and $\epsilon \to 0^+$ then yields
\begin{align}\label{liminf_I}
\liminf_{t\to\infty}\frac{I(t)}{t\,a(t)\,b(t)} \geq  \int_{0}^{1} \lambda^{\rho}(1-\lambda)^{\sigma}d \lambda.
\end{align}
By hypothesis there exists an increasing, $C^1$ function $\beta$ such that $b(t)/\beta(t) \to 1$ as $t\to\infty.$ It follows that there exists $T_1>0$ such that $t \geq T_1$ implies $b(t)/\beta(t) \leq 2$. Therefore with $\epsilon \in (0,\tfrac{1}{2})$, $t \geq 2T_1$ we have that $(1-\epsilon)t \geq T_1$. Suppose $t \geq 2T_1$ and estimate as follows
\[
I_1(t) = \int_0^{\epsilon t}a(s)\,b(t-s)\,ds \leq 2 \beta(t)\int_0^{\epsilon t} a(s)\,ds = 2 \beta(t)\,\epsilon\,t\,a(\epsilon t)\,\frac{\int_0^{\epsilon t}a(s)\,ds}{\epsilon\,t\,a(\epsilon t)}.
\]
Hence, for $t \geq 2T_1$,
\[
\frac{I_1(t)}{t\,a(t)\,b(t)} \leq 2\epsilon \,\frac{\beta(t)}{b(t)}\,\frac{a(\epsilon t)}{a(t)}\,\frac{\int_0^{\epsilon t}a(s)\,ds}{\epsilon\,t\,a(\epsilon t)}.
\]
$a \in \text{RV}_\infty(\rho)$ implies that $\lim_{t\to\infty}a(\epsilon t)/a(t) = \epsilon^{\rho}$ and similarly, by Karamata's Theorem,\\ $\lim_{t\to\infty}\int_0^{\epsilon t}a(s)ds/\epsilon\,t\,a(\epsilon t) = 1/(1+\rho)$. Thus
\begin{align}\label{limsup_I_1}
\limsup_{t\to\infty}\frac{I_1(t)}{t\,a(t)\,b(t)} \leq \frac{2 \epsilon^{\rho+1}}{1+\rho}.
\end{align}
Finally, consider $I_3(t)$. By construction $t \geq T_1$ implies $b(t)/\beta(t) \leq 2$ and since $b,\beta$ are continuous and positive, with $\beta$ bounded away from zero, $\sup_{0\leq t \leq T_1}b(t)/\beta(t) = \max_{0\leq t \leq T_1}b(t)/\beta(t) := B_1 < \infty$. Thus there exists $B_2 > 0$ such that $b(t) \leq B_2 \, \beta(t)$ for all $t \geq 0$. Therefore
\[
I_3(t) = \int_{(1-\eta)t}^t a(s)\,b(t-s)\,ds \leq B_2 \int_{(1-\eta)t}^t a(s)\,\beta(t-s)\,ds \leq B_2 \beta(\eta t) \int_{(1-\eta)t}^t a(s)\,ds.
\]
Hence
\begin{align}\label{est_I_3}
\limsup_{t\to\infty}\frac{I_3(t)}{t\,a(t)\,b(t)} \leq B_2\, \limsup_{t\to\infty}\frac{\beta(\eta t)}{b(t)}\,\limsup_{t\to\infty}\frac{\int_{(1-\eta)t}^t a(s)\,ds}{t\,a(t)} = B_2 \, \eta^{\sigma} \limsup_{t\to\infty}\frac{\int_{(1-\eta)t}^t a(s)\,ds}{t\,a(t)}.
\end{align}
The final limit on the right-hand side of \eqref{est_I_3} is calculated by once more calling upon the Uniform Convergence Theorem for Regularly Varying Functions
\[
\lim_{t\to\infty}\frac{\int_{(1-\eta)t}^t a(s)\,ds}{t\,a(t)} = \lim_{t\to\infty}\int_{1-\eta}^1 \frac{a(\lambda t)}{a(t)}d\lambda = \int_{1-\eta}^1 \lambda^{\rho}d\lambda.
\]
Returning to \eqref{est_I_3}
\begin{align}\label{limsup_I_3}
\limsup_{t\to\infty}\frac{I_3(t)}{t\,a(t)\,b(t)} \leq B_2\,\eta^{\sigma}\,\int_{1-\eta}^1 \lambda^{\rho}d\lambda = B_2\,\eta^{\sigma} \left( \frac{1}{\rho+1}-\frac{(1-\eta)^{\rho+1}}{\rho+1}\right).
\end{align}
Therefore, combining \eqref{lim_I_2}, \eqref{limsup_I_1} and \eqref{limsup_I_3}, we obtain
\begin{align*}
\limsup_{t\to\infty}\frac{I(t)}{t\,a(t)\,b(t)} \leq 2 \epsilon^{\rho+1}\, \frac{1}{1+\rho}+ \int_{\epsilon}^{1-\eta} \lambda^{\rho}(1-\lambda)^{\sigma}d \lambda + B_2\,\eta^{\sigma}\,\int_{1-\eta}^1 \lambda^{\rho}d\lambda.
\end{align*}
Letting $\eta$ and $\epsilon \to 0^+$ in the above then yields
\begin{align}\label{limsup_I}
\limsup_{t\to\infty}\frac{I(t)}{t\,a(t)\,b(t)} \leq \int_{0}^{1} \lambda^{\rho}(1-\lambda)^{\sigma}d \lambda.
\end{align}
Combining \eqref{limsup_I} with \eqref{liminf_I} gives the desired conclusion.
\end{proof}
The proof of Theorem \ref{main_thm} now begins in earnest by proving a ``rough'' lower bound on the solution which we will later refine. Lemmas \ref{lemma.lwr.bound}, \ref{limsup} and \ref{refine_liminf} are all proven under the same set of hypotheses and are presented separately purely for readability and clarity.
\begin{lemma}\label{lemma.lwr.bound}
Suppose the measure $\mu$ obeys \eqref{infinite_measure} with $M \in \text{RV}_\infty(\theta),\, \theta \geq 0$ and that $f \in \text{RV}_\infty(\beta),\,\beta\in [0,1)$. If $\beta=0$, let $f$ be asymptotically increasing and obey $\lim_{x\to\infty}f(x)=\infty$. Then the unique continuous solution, $x$, of \eqref{functional} obeys 
\[
\liminf_{t\to\infty}\frac{x(t)}{F^{-1}(t\,M(t))} > 0.
\]
\end{lemma}
\begin{proof} Let $\epsilon \in (0,1)$ be arbitrary. By hypothesis there exists $\phi$ such that \eqref{f_monotone_approx} holds and hence there exists $x_1(\epsilon)>0$ such that $f(x)>(1-\epsilon)\phi(x)$ for all $x>x_1(\epsilon)$. Furthermore, there exists $T_0(\epsilon)>0$ such that $t \geq T_0$ implies $x(t)>x_1(\epsilon)$. Similarly, there exists $T_1(\epsilon)>0$ such that $M(t) >0$ for all $t \geq T_1.$ Since $M \in \text{RV}_\infty(\theta)$, there exists a $C^1$ function $M_1$ such that for all $\epsilon\in(0,1)$ there exists $T_2(\epsilon)>0$ such that for all $t \geq T_2$, $M(t)>(1-\epsilon)M_1(t)$. Let $T_3 := T_0 + T_1 + T_2$. Hence, for $t \geq 4T_3$,  estimate as follows
\begin{align*}
x'(t) &= \int_{[0,t-T_3]}\mu(ds)f(x(t-s)) + \int_{(t-T_3,t]}\mu(ds)f(x(t-s))
\geq (1-\epsilon)\int_{[0,t-T_3]}\mu(ds)\phi(x(t-s)) \\
&= (1-\epsilon)\int_{[0,(t-T_3)/2]}\mu(ds)\phi(x(t-s)) + (1-\epsilon)\int_{((t-T_3)/2,t-T_3]}\mu(ds)\phi(x(t-s))\\
&\geq (1-\epsilon)\int_{[0,(t-T_3)/2]}\mu(ds)\phi(x(t-s))
\geq (1-\epsilon)M({\tfrac{1}{2}(t-T_3)})\phi\left( x(\tfrac{1}{2}(t+T_3)) \right). 
\end{align*} 
Since $M \in \text{RV}_\infty(\theta)$, $\lim_{t\to\infty}M({(t-T_3)/2})/M(t-T_3) = {2}^{-\theta}$. Thus there exists a positive constant $C$ and a time $\tilde{T_3}\geq 4T_3$ such that 
\begin{align}
x'(t) \geq C\,M(t-T_3)\phi\left( x(\tfrac{1}{2}(t+T_3)) \right), \mbox{ for all }t \geq \tilde{T_3}.
\end{align}
Furthermore, since $t\geq \tilde{T_3}$ implies $t-T_3>T_2$, there exists $C_0>0$ such that 
\begin{align}\label{diff_ineq_1}
x'(t) \geq C_0\,M_1(t-T_3)\phi\left( x((t+T_3)/2) \right), \mbox{ for all }t \geq \tilde{T_3}.
\end{align}
Now define the $C^2$, positive, increasing function $\bar{M}_1 (t) := \int_0^t M_1(s)ds$ for $t \geq 0$. Let 
\begin{align}\label{def.alpha}
\alpha(t) := \bar{M}_1^{-1}(t)+T_3, \quad t \geq \bar{M}_1(\tilde{T_3}).
\end{align}
For $t \geq \bar{M}_1(\tilde{T_3})$, $\alpha(t) \geq \alpha(\bar{M}_1(\tilde{T_3)}) = \tilde{T_3}+T_3 > \tilde{T_3}$ since $\alpha$ is increasing. Define $\tilde{x}(t) := x(\alpha(t))$ for $t \geq \bar{M}_1(\tilde{T_3})$. Note that $\tilde{x} \in C^1([\bar{M}_1(\tilde{T_3}),\infty);(0,\infty))$ and $\alpha'(t) = 1/M_1(\bar{M}_1^{-1}(t))$. For $t \geq \bar{M}_1(\tilde{T_3})$, use \eqref{diff_ineq_1} to compute
\begin{align}\label{diff_ineq_2}
\tilde{x}'(t) &= \alpha'(t)x'(\alpha(t)) 
\geq \frac{C_0\,M_1(\alpha(t)-T_3)}{M_1(\bar{M}_1^{-1}(t))}\phi\left( x(\tfrac{1}{2}(\alpha(t)+T_3)) \right) = C_0\,\phi\left( x(\tfrac{1}{2}(\alpha(t)+T_3))\right).
\end{align}
Define $\tau(t) = t - \bar{M}_1\left(\bar{M}_1^{-1}(t)/2\right)>0,$ for $t \geq \bar{M}_1(\tilde{T_3})$. It follows that $(\alpha(t)+T_3)/2 = \alpha(t-\tau(t))$. Hence, for $t \geq \bar{M}_1(\tilde{T_3})$
\begin{align}\label{diff_ineq_3}
\tilde{x}'(t) &\geq C_0\,\phi\left( x(\tfrac{1}{2}(\alpha(t)+T_3)\right) = C_0\,\phi\left( x(\alpha(t-\tau(t))\right) = C_0\,\phi\left(\tilde{x}(t-\tau(t))\right).
\end{align}
For $t \geq \bar{M}_1(\tilde{T_3})$ it is straightforward to show, using the monotonicity of $\bar{M}_1$, that $\tau(t)>0$. Using that $\bar{M}_1\in \text{RV}_\infty(\theta+1)$ we have
\begin{align*}
\lim_{t\to\infty}\frac{t-\tau(t)}{t} &= \lim_{t\to\infty}\frac{\bar{M}_1\left( \tfrac{1}{2}\bar{M}_1^{-1}(t)\right)}{\bar{M}_1\left( \bar{M}_1^{-1}(t)\right)} = \lim_{t\to\infty}\frac{\bar{M}_1\left( \tfrac{1}{2}\bar{M}_1^{-1}(t)\right)}{\bar{M}_1\left( \bar{M}_1^{-1}(t)\right)} = \left(\frac{1}{2}\right)^{\theta+1}.
\end{align*}
It follows that there exists $T_4>0$ such that for all $t \geq T_4$, $t-\tau(t) > 2^{-(\theta +2)}\,t$ for all $\epsilon>0$ sufficiently small. Letting $T_5 := \max(T_4, \bar{M}_1(\tilde{T_3}))$ we have, for $t \geq T_5$
\begin{align}\label{diff_ineq_4}
\tilde{x}'(t) \geq C_0\, \phi(\tilde{x}(qt)),\quad q = 2^{-(\theta +2)} \in (0,1).
\end{align} 
The following estimates will be needed to define a lower comparison solution. Since $\phi \circ \Phi^{-1}$ is in $\text{RV}_\infty\left(\beta/(1-\beta)\right)$ we have
\[
\lim_{x\to\infty}\frac{(\phi \circ \Phi^{-1})(\tfrac{x}{q})}{(\phi \circ \Phi^{-1})(x)} = \left(\frac{1}{q}\right)^\frac{\beta}{1-\beta}.
\]
Thus there exists $x_2>0$ such that for all $x\geq x_2$
\[
\frac{(\phi \circ \Phi^{-1})(\tfrac{x}{q})}{(\phi \circ \Phi^{-1})(x)} < 2\left(\frac{1}{q}\right)^\frac{\beta}{1-\beta}.
\]
Next let $T^{'}_5>0$ be so large that $\Phi(\tilde{x}(qT^{'}_5))- x_2 > 0$ and set $T_6 := \max(T_5,T^{'}_5)+1$. Then $\Phi(\tilde{x}(qT_6)) > \Phi(\tilde{x}(qT^{'}_5)) > x_2$. Define 
\begin{align}\label{c_constant}
c := \min\left(C_0\,\frac{q^{\tfrac{\beta}{1-\beta}}}{4}, \, \frac{\Phi(\tilde{x}(qT_6)) - x_2}{2T_6(1-q)}, \, \frac{\Phi(\tilde{x}(qT_6))}{2T_6} \right)
\end{align}
and 
\begin{align}\label{theta_constant}
d := cT_6 - \Phi(\tilde{x}(qT_6)).
\end{align}
Then define $x_0:= cqT_6 - d = \Phi(\tilde{x}(qT_6)) - cT_6(1-q)>x_2$. Note that $d< 0$ due to \eqref{c_constant}. Therefore $1/q-1>0$ and for any $x \geq x_0, \, x/q + \left(1/q-1\right)d < x/q$. Hence for $x \geq x_0$
\begin{align}\label{f_after_Finv}
\frac{(\phi \circ \Phi^{-1})\left(\tfrac{x}{q} + \left(\tfrac{1}{q}-1\right)d\right)}{(\phi \circ \Phi^{-1})(x)} \leq \frac{(\phi \circ \Phi^{-1})\left(\tfrac{x}{q}\right)}{(\phi \circ \Phi^{-1})(x)} < 2\left(\frac{1}{q}\right)^\frac{\beta}{1-\beta}.
\end{align}
Letting $t = (x+d)/cq$ in \eqref{f_after_Finv} and noting that \eqref{c_constant} implies $C_0/c \geq 4\left(1/q\right)^{\beta/(1-\beta)}$ we have
\begin{align}\label{C_0/c}
\frac{(\phi \circ \Phi^{-1})(ct - d)}{(\phi \circ \Phi^{-1})(cqt - d)} < 2\left(\frac{1}{q}\right)^\frac{\beta}{1-\beta} < \frac{C_0}{c}, \mbox{ for all }t \geq T_6.
\end{align}
Define the lower comparison solution, $x_-$, by
\begin{align}\label{lower_sol}
x_-(t) = \Phi^{-1}(ct - d), \,\, t \geq qT_6.
\end{align}
Then for $t \in [qT_6,\, T_6]$, by the monotonicity of $\Phi^{-1}$ and \eqref{theta_constant},
\[
x_-(t) \leq x_-(T_6) = \Phi^{-1}(cT_6 - d)= \tilde{x}(qT_6) \leq \tilde{x}(t).
\]
Also $x_-(T_6) = \tilde{x}(qT_6)< \tilde{x}(T_6)$, because $\tilde{x}$ is increasing. Hence 
\begin{align}\label{comparison}
x_-(t) < \tilde{x}(t), \,\, t \in [qT_6, T_6]. 
\end{align}
Next, since $\Phi(x_-(t)) = ct-d$, for $t \geq T_6$
\[
x_-'(t) = c\,\left(\phi \circ \Phi^{-1}\right)(ct-d) = c\,\phi(x_-(t)) = \frac{c}{C_0}\,\frac{\phi(x_-(t))}{\phi(x_-(qt))}\,C_0\,\phi(x_-(qt)).
\]
Now for $t \geq T_6$, by \eqref{C_0/c}
\[
\frac{c}{C_0}\frac{\phi(x_-(t))}{\phi(x_-(qt))} = \frac{c}{C_0}\frac{(\phi\circ\Phi^{-1})(ct-d)}{(\phi\circ\Phi^{-1})(cqt-d)} < \frac{c}{C_0}\frac{C_0}{c} = 1.
\]
Thus
\begin{align}\label{lower_diff}
x_-'(t) < C_0\, \phi(x_-(qt)), \,\, t \geq T_6.
\end{align}
Recalling \eqref{diff_ineq_4}, $\tilde{x}'(t) \geq C_0\,\phi(\tilde{x}(qt))$ for all $t \geq T_6 > T_5$. Then by \eqref{comparison} and \eqref{lower_diff}, because $\phi$ is increasing, $\tilde{x}(t) > x_-(t)$ for all $t \geq qT_6$. To see this suppose there is a minimal $t_0>T_6$ such that $x_-(t_0)=\tilde{x}(t_0)$. Thus $x_-'(t_0) \geq \tilde{x}'(t_0)$ and $x_-(t_0) < \tilde{x}(t)$ for all $t \in [qT_6, t_0)$. Then, since $t_0>T_6$ and $qt_0 > qT_6$, $\phi$ increasing yields
\[
\tilde{x}'(t_0) \geq C_0\,\phi(\tilde{x}(qt_0)) > C_0\, \phi(x_-(qt_0)) > x_-'(t_0) \geq \tilde{x}'(t_0),
\]  
a contradiction. Now, for $t \geq qT_6$, $\tilde{x}(t) > x_-(t) = \Phi^{-1}(ct-d)$. Hence for $t \geq qT_6$
\[
x(\alpha(t)) = \tilde{x}(t) > \Phi^{-1}(ct-d).
\]
From the definition of $\alpha$, in \eqref{def.alpha}, $\alpha^{-1}(t) = \bar{M}_1(t-T_3)$ and therefore
\[
x(t) = \tilde{x}(\alpha^{-1}(t)) > \Phi^{-1}(c\,\alpha^{-1}(t)-d) = \Phi^{-1}(c\,\bar{M}_1(t-T_3)-d), \quad \bar{M}_1(t-T_3)> qT_6.
\]
Hence, recalling that $d<0$,
\begin{align}\label{Phi_x_lower}
\Phi(x(t)) > c \bar{M}_1(t-T_3)-d > c \bar{M}_1(t-T_3) , \quad \bar{M}_1(t-T_3)> qT_6.
\end{align}
Note that for $t>2T_3$, $t/2 < t-T_3.$ Since $\bar{M}_1$ is increasing this implies that $\bar{M}_1(t/2) \leq \bar{M}_1(t-T_3)$. Thus \eqref{Phi_x_lower} implies
\[
\liminf_{t\to\infty}\frac{\Phi(x(t))}{\bar{M}_1(t)} \geq \liminf_{t\to\infty}\frac{c\,\bar{M}_1\left(\frac{t}{2}\right)}{\bar{M}_1(t)} = c\, 2^{-(\theta +1)} > 0.
\]
By Karamata's Theorem $\lim_{t\to\infty}\bar{M}_1(t)/tM_1(t) = 1/(1+\theta)$ and therefore
\[
\liminf_{t\to\infty}\frac{\Phi(x(t))}{t\,M_1(t)}\geq c\,(1+\theta)\, 2^{-(\theta +1)} > 0.
\]
Finally, since $\Phi^{-1} \in \text{RV}_\infty(1/(1-\beta))$ and $M$ is asymptotic to $M_1$, we conclude that
\[
\liminf_{t\to\infty}\frac{x(t)}{\Phi^{-1}(t\,M(t))}>0,
\]
as required.
\end{proof}
\begin{lemma}\label{limsup}
Suppose the hypotheses of Lemma \ref{lemma.lwr.bound} hold. Then the unique continuous solution, $x$, of \eqref{functional} obeys 
\[
\limsup_{t\to\infty}\frac{F(x(t))}{t\,M(t)} \leq \frac{1}{1-\beta}B\left(\theta+1,\frac{\theta \beta+1}{1-\beta}\right).
\]
\end{lemma}
\begin{proof}
Once again let $\phi$ satisfying \eqref{f_monotone_approx} obey $f(x)/\phi(x) < (1+\epsilon)$ for all $x > x_1(\epsilon)$, for any $\epsilon>0$ and for some $x_1(\epsilon)>0$. Owing to the fact that $\lim_{t\to\infty}x(t)=\infty$ there exists $T_1(\epsilon)$ such that $t\geq T_1(\epsilon)$ implies $x(t) > x_1(\epsilon)$. Since $\lim_{t\to\infty}M(t) = \infty$ there exists $T_2(\epsilon)$ such that $M(t) > 0$ for all $t\geq T_2$. Hence, for all $t \geq 2\,\max(T_1,T_2)$, \eqref{volterra} becomes
\begin{align}\label{upper_1}
\frac{x(t)}{\phi(x(t))} \leq \frac{x(0)}{\phi(x(t))} + \frac{\int_0^{T_1} M(t-s)f(x(s))\,ds}{\phi(x(t))} +(1+\epsilon)\,t\,M(t),
\end{align}
where the upper bound on the term $\int_{T_1}^t M(t-s) \phi(x(s))\,ds$ was obtained by exploiting the fact that $t \mapsto x(t)$ and $t \mapsto M(t)$ are non-decreasing. By Karamata's Theorem and the regular variation of $\phi$, it is true that $\lim_{x\to\infty}(1-\beta)\phi(x)\Phi(x)/x = 1$. Thus for all $\epsilon>0$ there exists $x_2(\epsilon)$ such that
\[
\Phi(x) < \frac{(1+\epsilon)x}{(1-\beta)\phi(x)}, \,\, \mbox{ for all } x > x_2(\epsilon).
\] 
Once more the divergence of $x(t)$ yields the existence of a $T_3(\epsilon)$ such that $x(t) > x_2(\epsilon)$ for all $t \geq T_3(\epsilon)$. Letting $T_4 = 2\,\max(T_1,T_2,T_3)$ we obtain
\[
\frac{\Phi(x(t))}{t\,M(t)} < \frac{(1+\epsilon)x(t)}{(1-\beta)\phi(x(t))\,t \, M(t)}, \,\, \mbox{ for all } t \geq T_4.
\] 
Combining the above estimate with \eqref{upper_1} yields
\[
\frac{\Phi(x(t))}{t\,M(t)} < \frac{(1+\epsilon) x(0)}{(1-\beta)\phi(x(t))\,t\,M(t)} + \frac{(1+\epsilon)\int_0^{T_1} M(t-s)f(x(s))\,ds}{(1-\beta)\phi(x(t))\,t\,M(t)} + \frac{(1+\epsilon)^2}{(1-\beta)}, \,t \geq T_4(\epsilon).
\]
Hence, letting $t\to\infty$ and then sending $\epsilon \to 0^+$, we get
\begin{align*}
\limsup_{t\to\infty}\frac{\Phi(x(t))}{t\,M(t)} \leq \frac{1}{1-\beta}.
\end{align*}
Since $\Phi^{-1} \in \text{RV}_\infty\left(1/(1-\beta)\right)$ the above estimate can be restated as 
\[
\limsup_{t\to\infty}\frac{x(t)}{\Phi^{-1}(t\,M(t))} \leq  \left(1-\beta\right)^{\tfrac{1}{\beta-1}}  < \infty.
\]
We now seek to refine the ``crude'' upper bound on the growth of the solution obtained above. From the above construction and Lemma \ref{lemma.lwr.bound} we may suppose that
\begin{align}\label{crude_limsup}
\limsup_{t\to\infty} \frac{x(t)}{\Phi^{-1}(t\,M(t))} =: \eta \in (0,\infty).
\end{align} 
From \eqref{crude_limsup} it follows that for all $\epsilon>0$ there exists $T_5(\epsilon)>0$ such that for all $t \geq T_5(\epsilon)$,
$x(t) <(\eta+\epsilon)\Phi^{-1}(t\,M(t))$. By monotonicity of $\phi$ it follows that
\[
\frac{\phi(x(t))}{\phi\left( \Phi^{-1}(t\,M(t)) \right)} < \frac{\phi \left((\eta+\epsilon)\Phi^{-1}(t\,M(t)) \right)}{\phi\left( \Phi^{-1}(t\,M(t)) \right)}, \quad t \geq T_5(\epsilon).
\]
Since $\phi\in \text{RV}_\infty(\beta)$
\[
\limsup_{t\to\infty}\frac{\phi(x(t))}{\phi\left( \Phi^{-1}(t\,M(t)) \right)} \leq (\eta+\epsilon)^\beta.
\]
Thus for all $\epsilon>0$ there exists $T_6(\epsilon)>0$ such that for all $t \geq T_6$,
\[
\phi(x(t)) <(1+\epsilon)(\eta+\epsilon)^\beta \phi\left(\Phi^{-1}(t\,M(t)) \right).
\]
Integrating this estimate yields
\begin{align}\label{int_upper_est}
\int_{T_6}^t M(t-s) \phi(x(s))\,ds \leq (1+\epsilon)(\eta+\epsilon)^\beta\int_{T_6}^t M(t-s)\phi\left(\Phi^{-1}(s\,M(s)) \right)ds, \quad t \geq T_6(\epsilon).
\end{align}
Since $(\phi\ \circ \Phi^{-1})(t\,M(t)) \in \text{RV}_\infty\left(\beta(1+\theta)/(1-\beta) \right)$ and $M \in \text{RV}_\infty(\theta)$, Lemma \ref{Beta_Lemma} can be applied to obtain
\begin{align}\label{apply_lemma}
\lim_{t\to\infty}\frac{\int_{0}^t M(t-s)\phi\left(\Phi^{-1}(s\,M(s)) \right)ds}{t\,M(t)\,\phi\left(\Phi^{-1}(t\,M(t)) \right)} = B\left(\theta+1,\frac{\theta \beta+1}{1-\beta}\right).
\end{align}
Hence combining \eqref{int_upper_est} and \eqref{apply_lemma} yields
\[
\limsup_{t\to\infty}\frac{\int_{T_6}^t M(t-s) \phi(x(s))\,ds}{t\,M(t)\,\phi\left(\Phi^{-1}(t\,M(t)) \right)} \leq (1+\epsilon)(\eta+\epsilon)^\beta B\left(\theta+1,\frac{\theta \beta+1}{1-\beta}\right).
\]
Apply the above estimate to \eqref{volterra} as follows
\begin{align*}
\eta &= \limsup_{t\to\infty}\frac{x(t)}{\Phi^{-1}(t\,M(t))} \leq \limsup_{t\to\infty}\frac{\int_0^{T_6}M(t-s)f(x(s))\,ds}{\Phi^{-1}(t\,M(t))} + \limsup_{t\to\infty}\frac{(1+\epsilon)\int_{T_6}^t M(t-s)\phi(x(s))\,ds}{\Phi^{-1}(t\,M(t))}\\
&\leq (1+\epsilon)^2(\eta+\epsilon)^\beta B\left(\theta+1,\frac{\theta \beta+1}{1-\beta}\right) \limsup_{t\to\infty}\frac{t\,M(t)\,\phi\left(\Phi^{-1}(t\,M(t)) \right)}{\Phi^{-1}(t\,M(t))}\\
&= (1+\epsilon)^2(\eta+\epsilon)^\beta B\left(\theta+1,\frac{\theta \beta+1}{1-\beta}\right) \limsup_{x\to\infty}\frac{x\, \phi\left(\Phi^{-1}(x) \right)}{\Phi^{-1}(x)}.
\end{align*}
Letting $\epsilon\to 0^+$ and using Karamata's Theorem to the remaining limit on the right-hand side
\begin{align*}
\eta^{1-\beta}
= \limsup_{y\to\infty}\frac{\Phi(y)\phi(y)}{y}B\left(\theta+1,\frac{\theta \beta+1}{1-\beta}\right)
= \frac{1}{1-\beta}B\left(\theta+1,\frac{\theta \beta+1}{1-\beta}\right),
\end{align*}
with $y = \Phi^{-1}(x)$ so that $y \to \infty$ as $x\to\infty$. Thus 
\[
\eta = \limsup_{t\to\infty}\frac{x(t)}{\Phi^{-1}(t\,M(t))} \leq \left\{\frac{1}{1-\beta}B\left(\theta+1,\frac{\theta \beta+1}{1-\beta}\right) \right\}^{\tfrac{1}{1-\beta}}.
\]
Using $\Phi \in \text{RV}_\infty(1-\beta)$ and $\Phi(x) \sim F(x)$ as $x\to\infty$ the above upper bound can be reformulated as 
\[
\limsup_{t\to\infty}\frac{F(x(t))}{t\,M(t)} \leq \frac{1}{1-\beta}B\left(\theta+1,\frac{\theta \beta+1}{1-\beta}\right),
\]
which is the required estimate.
\end{proof}
\begin{lemma}\label{refine_liminf}
Suppose the hypotheses of Lemma \ref{lemma.lwr.bound} hold. Then the unique continuous solution, $x$, of \eqref{functional} obeys 
\[
\liminf_{t\to\infty}\frac{F(x(t))}{t\,M(t)} \geq \frac{1}{1-\beta}B\left(\theta+1,\frac{\theta \beta+1}{1-\beta}\right).
\]
\end{lemma}
\begin{proof}
By Lemma \ref{lemma.lwr.bound} and Lemma \ref{limsup} 
\[
\liminf_{t\to\infty}\frac{x(t)}{\Phi^{-1}(t\,M(t))} =: \eta \in (0,\infty).
\]
Then for all $\epsilon \in (0,\eta)\cap (0,1)$ there exists $T_1(\epsilon)>0$ such that for all $t \geq T_1$
$\eta - \epsilon < x(t) / \Phi^{-1}(t\,M(t))$. Since $\lim_{t\to\infty}M(t)=\infty$ there exists $T_2$ such that $M(t)>0$ for all $t \geq T_2$. Hence
\begin{align}\label{1st_est}
x(t) > (\eta-\epsilon)\Phi^{-1}\left(t\,M(t)\right), \quad t \geq T_3:= \max(T_1,T_2).
\end{align}
Using monotonicity and regular variation of $\phi$ it follows from \eqref{1st_est} that
\[
\liminf_{t\to\infty}\frac{\phi(x(t))}{(\phi\circ\Phi^{-1})(t\,M(t))}\geq (\eta-\epsilon)^\beta.
\]
Now, because $\phi(x) \sim f(x)$ as $x\to\infty$, for all $\epsilon \in (0,\eta)\cap (0,1)$ there exists $T_4(\epsilon)>0$ such that 
\[
f(x(t))> (1-\epsilon)\phi(x(t)) > (1-\epsilon)^2(\eta-\epsilon)^\beta\, (\phi\circ\Phi^{-1})(t\,M(t)), \quad t \geq T_4(\epsilon).
\]
Integration then yields
\[
\int_0^t M(t-s)f(x(s))\,ds > (1-\epsilon)^2(\eta-\epsilon)^\beta\, \int_{T_4}^t M(t-s)(\phi\circ\Phi^{-1})(s\,M(s))\,ds.
\]
Hence, as in the proof of Lemma \ref{limsup}, applying Lemma \ref{Beta_Lemma} gives
\begin{align}\label{int_low_est}
\liminf_{t\to\infty}\frac{\int_0^t M(t-s)f(x(s))\,ds}{t\,M(t)\,(\phi\circ\Phi^{-1})(t\,M(t))} \geq (1-\epsilon)^2(\eta-\epsilon)^\beta \, B\left(\theta+1,\frac{\theta \beta+1}{1-\beta}\right).
\end{align}
Now apply the estimate from \eqref{int_low_est} to \eqref{volterra} as follows
\begin{align*}
\eta &= \liminf_{t\to\infty}\frac{x(t)}{\Phi^{-1}(t\,M(t))} \geq \liminf_{t\to\infty}\frac{\int_0^t M(t-s)f(x(s))\,ds}{\Phi^{-1}(t\,M(t))}\\ &= (1-\epsilon)^2(\eta-\epsilon)^\beta \, B\left(\theta+1,\frac{\theta \beta+1}{1-\beta}\right) \liminf_{t\to\infty}\frac{t\,M(t)\,(\phi\circ\Phi^{-1})(t\,M(t))}{\Phi^{-1}(t\,M(t))} \\
&= (1-\epsilon)^2(\eta-\epsilon)^\beta \, B\left(\theta+1,\frac{\theta \beta+1}{1-\beta}\right) \liminf_{x\to\infty}\frac{x\, \phi\left(\Phi^{-1}(x) \right)}{\Phi^{-1}(x)}.
\end{align*}
The limit of the final term on the right-hand side is $1/(1-\beta)$ by Karamata's Theorem and sending $\epsilon\to 0^+$ yields
\[
\eta = \frac{\eta^\beta}{1-\beta} B\left(\theta+1,\frac{\theta \beta+1}{1-\beta}\right).
\]
Hence
\[
\liminf_{t\to\infty}\frac{x(t)}{F^{-1}(t\,M(t))} \geq \left\{\frac{1}{1-\beta}B\left(\theta+1,\frac{\theta \beta+1}{1-\beta}\right) \right\}^{\frac{1}{1-\beta}}.
\]
Since $F\in \text{RV}_\infty(1-\beta)$ this can be rewritten in the form
\[
\liminf_{t\to\infty}\frac{F(x(t))}{t\,M(t)} \geq \frac{1}{1-\beta}B\left(\theta+1,\frac{\theta \beta+1}{1-\beta}\right),
\]
which is the desired bound.
\end{proof}
As with Theorem \ref{main_thm}, the proof of Theorem \ref{thm.pert} is split into a series of lemmata. A final consolidating argument then establishes the result as stated in Section \ref{results}.
\begin{lemma}\label{lemma.liminf.pert}
Suppose the measure $\mu$ obeys \eqref{infinite_measure} with $M \in \text{RV}_\infty(\theta),\, \theta \geq 0$ and that $f \in \text{RV}_\infty(\beta),\,\beta\in [0,1)$. If $\beta=0$, let $f$ be asymptotically increasing and obey $\lim_{x\to\infty}f(x)=\infty$. Let $x(t)$ denote the unique continuous solution of \eqref{functional_pert} and suppose $H \in C((0,\infty);(0,\infty))$. Then
\begin{align}\label{liminf.pert}
\liminf_{t\to\infty}\frac{x(t)}{F^{-1}(t\,M(t))} \geq L := \left\{ \frac{1}{1-\beta}B\left(1+\theta, \frac{1+\theta \beta}{1-\beta}\right) \right\}^{\tfrac{1}{1-\beta}} > 0.
\end{align}
\end{lemma}
\begin{proof}
With $\epsilon \in (0,1)$ arbitrary and $T_0(\epsilon)$ and $T_1(\epsilon)$ defined as in Lemma \ref{lemma.lwr.bound}, \eqref{functional_pert} admits the initial lower estimate
\[
x(t) > x(0) + H(t) + (1-\epsilon)\int_T^t M(t-s) \phi(x(s))\,ds, \quad t \geq T(\epsilon):=T_0(\epsilon)+T_1(\epsilon).
\] 
Letting $y(t) = x(t+T)$ and noting that $H(t)>0$ for $t>0$ we get
\begin{align*}
y(t) &> x(0) + (1-\epsilon)\int_T^{t+T} M(t+T-s) \phi(x(s))\,ds
= x(0) + (1-\epsilon)\int_0^{t} M(t-u) \phi(x(u+T))\,du\\
&= x(0) + (1-\epsilon)\int_0^{t} M(t-u) \phi(y(u))\,du, \quad t \geq T(\epsilon).
\end{align*}
Now consider the comparison equation defined by
\begin{align}\label{eq.perturbed_comparison}
x_\epsilon'(t) = (1-\epsilon)\int_{[0,t]} \mu(ds) \phi(x_\epsilon(t-s)),\quad t >0, \quad x_\epsilon(0) = x(0)/2.
\end{align}
In contrast to \eqref{functional_pert}, the solution to \eqref{eq.perturbed_comparison} will be non-decreasing. Integrating \eqref{eq.perturbed_comparison} using Fubini's Theorem yields
\[
x_\epsilon(t) = x(0)/2 + (1-\epsilon)\int_0^t M(t-u) \phi(x_\epsilon(u))du, \quad t \geq 0.
\]
By construction $x_\epsilon(t) < y(t) = x(t+T)$ for all $t \geq 0$, or $x(t) > x_\epsilon(t-T)$ for all $t \geq T$. Applying Theorem \ref{main_thm} to $x_\epsilon$ then yields
\[
\lim_{t\to\infty}\frac{F(x_\epsilon(t))}{t\, M_\epsilon(t)} = \frac{1}{1-\beta}B\left(1+\theta,\frac{1+\theta\beta}{1-\beta}  \right),
\]  
where $M_\epsilon(t) = (1-\epsilon)M(t)$. Hence
\[
\lim_{t\to\infty}\frac{F(x_\epsilon(t))}{t\, M(t)} = \frac{1-\epsilon}{1-\beta}B\left(1+\theta,\frac{1+\theta\beta}{1-\beta}  \right).
\]
Therefore
\begin{align*}
\liminf_{t\to\infty}\frac{F(x(t))}{t\, M(t)} &\geq \liminf_{t\to\infty}\frac{F(x_\epsilon(t-T))}{t\, M(t)} = \liminf_{t\to\infty}\frac{F(x_\epsilon(t-T))}{(t-T) M(t-T)}\frac{(t-T) M(t-T)}{t\,M(t)}\\
&= \frac{1-\epsilon}{1-\beta}B\left(1+\theta,\frac{1+\theta\beta}{1-\beta}  \right),
\end{align*}
where the final equality follows from the trivial fact that $t-T \sim t$ as $t\to\infty$ and noting that $M$ preserves asymptotic equivalence because $M \in \text{RV}_\infty(\theta)$. Finally, letting $\epsilon\to 0^+$ and using the regular variation of $F^{-1}$ yields
\[
\liminf_{t\to\infty}\frac{x(t)}{F^{-1}(t\, M(t))} \geq \left\{\frac{1}{1-\beta}B\left(1+\theta,\frac{1+\theta\beta}{1-\beta}  \right)\right\}^{\tfrac{1}{1-\beta}} = L,
\]
which finishes the proof.
\end{proof}
\begin{lemma}\label{lemma.limsup.pert}
Suppose the hypotheses of Lemma \ref{lemma.liminf.pert} hold and $\lim_{t\to\infty}H(t)/F^{-1}(t\,M(t)) = \lambda \in [0,\infty)$. Then, with $x$ denoting the unique continuous solution of \eqref{functional_pert}, 
\begin{align}\label{limsup.pert}
\limsup_{t\to\infty}\frac{x(t)}{F^{-1}(t\,M(t))} \leq U := \left(\frac{\lambda}{L^\beta}+\frac{1}{1-\beta} \right)^{\tfrac{1}{1-\beta}},
\end{align}
where $L$ is defined by \eqref{liminf.pert}.
\end{lemma}
\begin{proof}
We begin by constructing a monotone comparison solution which will majorise the solution of \eqref{functional_pert} and to which Lemma \ref{lemma.liminf.pert} can be applied. Let $\epsilon \in (0,1)$ be arbitrary and define $T_1(\epsilon)$ and $T_2(\epsilon)$ as in the proof of Lemma \ref{limsup}.

By hypothesis $\lim_{t\to\infty}H(t)/F^{-1}(t\,M(t)) = \lambda \in [0,\infty)$ and so there exists a $T(\epsilon)>0$ such that $t \geq T(\epsilon)$ implies $H(t) < (\lambda+\epsilon)\Phi^{-1}(t\,M(t))$. $M \in \text{RV}_\infty(\theta)$ implies there exists $M_1 \in C^1$ asymptotic to $M$ and $T_0(\epsilon)>T$ such that $M(t)<(1+\epsilon)M_1(t)$ for all $t \geq T_0$. For $t \geq T_0$, because $\Phi^{-1}$ is increasing, $\Phi^{-1}(t\,M(t)) < \Phi^{-1}(t\,(1+\epsilon)\,M_1(t))$ and since $\Phi^{-1} \in \text{RV}_\infty(1/(1-\beta))$ there exists $T^*>T_0$ such that $\Phi^{-1}(t\,M(t)) < (1+\epsilon)^{(2-\beta)/(1-\beta)}\Phi^{-1}(t\,M_1(t))$ for all $t \geq T^*$. 

For notational convenience define the quantity $\epsilon^*$ by letting $(1+\epsilon^*):=(1+\epsilon)^{(2-\beta)/(1-\beta)}$; note that $(1+\epsilon^*)\to 1$ as $\epsilon\to 0^+$. Defining $T_2' := T^* + T_1 + T_2$, we have the estimate
\begin{align}\label{initial_upper_pert}
x(t) &< x(0) + H(t) + \int_0^{T_2'} M(t-s)f(x(s))\,ds +(1+\epsilon)\int_{T_2'}^t M(t-s)\phi(x(s))\,ds\nonumber\\
&\leq x(0) + H(t) + M(t)\,T_2'\,F^*+(1+\epsilon)\int_{T_2'}^t M(t-s)\phi(x(s))\,ds\nonumber\\
&< x(0) + (\lambda+\epsilon)(1+\epsilon^*)\Phi^{-1}(t\,M_1(t)) + (1+\epsilon)M_1(t)\,T_2'\,F^*+(1+\epsilon)\int_{T_2'}^t M(t-s)\phi(x(s))\,ds,
\end{align}
for all $t \geq T_2'$ and where $F^* := \max_{0\leq s \leq T_2'}f(x(s))$. Now define the constant $x^*:=\max_{0 \leq s \leq T_2'}x(s)$ and the function 
\[
\bar{H}(t) := (\lambda+\epsilon)(1+\epsilon^*)\Phi^{-1}(t\,M_1(t)) + (1+\epsilon)M_1(t)\,T_2'\,F^* - (\lambda+\epsilon)(1+\epsilon^*),\quad t \geq 0 .
\]
Since $\Phi^{-1}(0)=1$ and $M_1(0)=0$, $H(0)=0$ and by construction $H \in C^1((0,\infty);(0,\infty))$. The initial upper estimate \eqref{initial_upper_pert} motivates the definition of the following upper comparison equation:
\[
y_\epsilon '(t) :=  \bar{H}'(t)+(1+\epsilon)\int_{[0,t]} \mu(ds)\phi(y_\epsilon(t-s))\,ds,\quad t \geq 0,\quad y_\epsilon(0) = x(0) + x^* + (\lambda+\epsilon)(1+\epsilon^*).
\]
Integration using Fubini's theorem quickly shows that 
\[
y_\epsilon(t) = x(0)+x^* + (\lambda+\epsilon)(1+\epsilon^*) + \bar{H}(t) + (1+\epsilon)\int_0^t M(t-s) \phi(y_\epsilon(s))\,ds, \quad t \geq 0.
\]
Since $y_\epsilon(t)$ is non-decreasing it is immediately clear that $x(t) \leq y_\epsilon(t)$ for all $t \in [0,T_2']$. A simple time of the first breakdown argument using the estimate \eqref{initial_upper_pert} then yields that $x(t) \leq y_\epsilon(t)$ for all $t \geq 0$. We now compute an explicit upper bound on $\limsup_{t\to\infty}y_\epsilon(t)/F^{-1}(t\,M(t))$. Monotonicity readily yields
\[
y_\epsilon(t) \leq x(0) + x^* + (\lambda+\epsilon)(1+\epsilon^*)\Phi^{-1}(t\,M_1(t)) + (1+\epsilon)M_1(t)\,T_2'\,F^*+(1+\epsilon)M(t)\,t\,\phi(y_\epsilon(t)), \quad t \geq 0.
\]
Hence, with $C(t)$ suitably defined,
\[
\frac{y_\epsilon(t)}{t\,M(t)\,\phi(y_\epsilon(t))} \leq C(t) + \frac{(\lambda+\epsilon)(1+\epsilon^*)\Phi^{-1}(t\,M_1(t))}{t\,M(t)\,\phi(y_\epsilon(t))} +(1+\epsilon),\quad t \geq 0.
\]
A short calculation reveals that $\lim_{t\to\infty}C(t) = 0$. By Karamata's Theorem there exists a $T_3(\epsilon)$ such that
\begin{align}\label{upper.est.pert}
\frac{\Phi(y_\epsilon(t))}{t\,M(t)} < \frac{(1+\epsilon)C(t)}{1-\beta} + \frac{(1+\epsilon)(\lambda+\epsilon)(1+\epsilon^*)\Phi^{-1}(t\,M_1(t))}{(1-\beta)\,t\,M(t)\,\phi(y_\epsilon(t))} +\frac{(1+\epsilon)^2}{1-\beta}, \quad t \geq T_4:= T_3+T_2'.
\end{align}
By applying Lemma \ref{lemma.liminf.pert} to $y_\epsilon$ we conclude that 
\[
\liminf_{t\to\infty}\frac{y_\epsilon(t)}{\Phi^{-1}(t\,M(t))} =: L \in (0,\infty].
\]
If $L \in (0,\infty)$ then there exists a $T_5(\epsilon)$ such that for all $t \geq T_6:= T_5+T_4$
\begin{align}\label{final.est}
\frac{\Phi(y_\epsilon(t))}{t\,M(t)} &< \frac{(1+\epsilon)C(t)}{1-\beta} + \frac{(1+\epsilon)(\lambda+\epsilon)(1+\epsilon^*)\Phi^{-1}(t\,M(t))}{(1-\beta)\,t\,M(t)\,\phi((1-\epsilon)\,L\,\Phi^{-1}(t\,M(t)))} +\frac{(1+\epsilon)^2}{1-\beta}\nonumber\\
&< \frac{(1+\epsilon)C(t)}{1-\beta} + \frac{(1+\epsilon)(\lambda+\epsilon)(1+\epsilon^*)\Phi^{-1}(t\,M(t))}{(1-\beta)\,t\,M(t)\,(1-\epsilon)^\beta \, L^\beta\,\phi(\Phi^{-1}(t\,M(t)))} +\frac{(1+\epsilon)^2}{1-\beta}.
\end{align}
By Karamata's Theorem the following asymptotic equivalence holds
\[
(1-\beta)\,t\,M(t)\,\phi\left(\Phi^{-1}(t\,M(t))\right) \sim \Phi^{-1}(t\,M(t)) \mbox{ as }t\to\infty.
\]
Therefore taking the limit superior across \eqref{final.est} yields
\[
\limsup_{t\to\infty}\frac{\Phi(y_\epsilon(t))}{t\,M(t)} \leq \frac{(1+\epsilon)(\lambda+\epsilon)(1+\epsilon^*)}{(1-\epsilon)^\beta \, L^\beta}+\frac{(1+\epsilon)^2}{1-\beta}.
\]
By letting $\epsilon\to 0^+$ and using the regular variation of $\Phi^{-1}$
\[
\
	\limsup_{t\to\infty}\frac{x(t)}{\Phi^{1}(t\,M(t))} \leq \left(\frac{\lambda}{L^\beta}+\frac{1}{1-\beta} \right)^{\tfrac{1}{1-\beta}}=: U.
\]
If $L := \liminf_{t\to\infty}y_\epsilon(t)/\Phi^{-1}(t\,M(t))=\infty$ the above construction will yield 
\[
\limsup_{t\to\infty}y_\epsilon(t)/\Phi^{-1}(t\,M(t))<\infty,
\] 
a contradiction. Hence $L \in (0,\infty)$ and the claim is proven.
\end{proof}
\begin{lemma}\label{iteration}
Suppose $\beta \in [0,1)$, $\lambda \in [0,\infty)$ and consider the iterative scheme defined by 
\begin{align}\label{scheme}
x_{n+1} = g(x_n) := \frac{x_n^\beta}{1-\beta} B\left(1+\theta, \frac{1+\theta \beta}{1-\beta}\right) + \lambda, \,\, n \geq 1;\,\, x_0 \in [L, C^*],
\end{align}
with $L$ defined by \eqref{liminf.pert}, $U$ defined by \eqref{limsup.pert} and
\begin{align}\label{C_star}
C^* := \max\left(U,\, L + \frac{\lambda}{1-\beta} \right).
\end{align}
Then there exists a unique $x_\infty \in [L,C^*]$ such that $\lim_{n\to\infty}x_n = x_\infty$.
\end{lemma}
\begin{proof}
By inspection, $g \in C([L,\infty);(0,\infty))$. We calculate as follows
\[
g'(x) = \frac{\beta}{1-\beta}x^{\beta-1}B\left(1+\theta, \frac{1+\theta \beta}{1-\beta}\right) > 0, \quad x >0,
\]
and similarly
\[
g''(x) = -\beta x^{\beta-2}B\left(1+\theta, \frac{1+\theta \beta}{1-\beta}\right) < 0, \quad x >0.
\]
Therefore $g'(L) = \beta > g'(x) > 0$ for all $x>L$ and $|g'(x)| \leq \beta <1$ for all $x \in [L,\infty)$. Since $g$ is monotone increasing it is sufficient check that $g$ maps $[L,\,C^*]$ to $[L,\,C^*]$ as follows. Firstly,
\begin{align}\label{g_of_L}
g(L) = \frac{L^\beta}{1-\beta}B\left(1+\theta, \frac{1+\theta\beta}{1-\beta}\right)+\lambda = L + \lambda \in [L,\,C^*].
\end{align}
By the Mean Value Theorem there exists $\xi \in [L, C^*]$ such that 
\[
\frac{g(C^*) - g(L)}{C^* - L} = g'(\xi) \leq \beta.
\]
Therefore $g(C^*) \leq \beta (C^* - L) + g(L)$ and thus a sufficient condition for $g(C^*) \leq C^*$ is $\beta (C^* - L) + g(L) \leq C^*$ or $C^* \geq (g(L) - L\beta)/(1-\beta) = L + \lambda/(1-\beta)$, using \eqref{g_of_L}. Thus with $C^*$ as defined in \eqref{C_star}, $g: [L,C^*] \to [L,C^*]$. Hence 
\eqref{scheme} has a unique fixed point in $[L,\, C^*]$ and the claim follows.
\end{proof}
With the preceding auxiliary results proven we are now in a position to supply the proof of Theorem \ref{thm.pert}, as promised.
\begin{proof}[Proof of Theorem \ref{thm.pert}]
Suppose that $(ii.)$ holds, or that $\lim_{t\to\infty}H(t)/F^{-1}(t\,M(t)) = \lambda \in [0,\infty)$. The idea here is to combine the crude bounds on the solution from Lemmas \ref{lemma.liminf.pert} and \ref{lemma.limsup.pert} with a fixed point argument based on Lemma \ref{iteration} to complete the proof that $(ii.)$ implies $(i.)$. We compute 
$\limsup_{t\to\infty}x(t)/F^{-1}(t\,M(t))$ in detail only as the calculation of the corresponding limit inferior proceeds in an analogous manner.
To begin make the following induction hypothesis 
\[
\left(H_n\right)\quad\quad\limsup_{t\to\infty}\frac{x(t)}{\Phi^{-1}(Mt)} \leq \zeta_n, \quad \zeta_{n+1} := \frac{\zeta_n^\beta}{1-\beta}B\left(1+\theta,\frac{1+\theta\beta}{1-\beta}\right) + \lambda,\quad n \geq 0,
\]
and choose $\zeta_0 := U$. $\left(H_0\right)$ is true by Lemma \ref{lemma.limsup.pert}. Suppose that $\left(H_n\right)$ holds. Thus there exists $T(\epsilon)>0$ such that $x(t)< (\zeta_n+\epsilon)\Phi^{-1}(t\,M(t))$ for all $t \geq T$. Hence 
\[
\frac{\phi(x(t))}{\phi(\Phi^{-1}(t\,M(t)))} < \frac{\phi((\zeta_n+\epsilon)\Phi^{-1}(Mt))}{\phi(\Phi^{-1}(t\,M(t)))}, \,\, t \geq T.
\]
The regular variation of $\phi$ thus yields $\limsup_{t\to\infty}\phi(x(t))/\phi(\Phi^{-1}(t\,M(t))) \leq (\zeta_n+\epsilon)^\beta$. Therefore there exists a $T_2(\epsilon)>0$ such that $t \geq T_2$ implies $f(x(t))  < (1+\epsilon)[(\zeta_n+\epsilon)^\beta + \epsilon]\phi(\Phi^{-1}(t\,M(t)))$. From \eqref{volterra_pert}
\[
\limsup_{t\to\infty}\frac{x(t)}{\Phi^{-1}(t\,M(t))} = \limsup_{t\to\infty}\frac{\int_0^t M(t-s)f(x(s))ds}{\Phi^{-1}(t\,M(t))} + \lim_{t\to\infty}\frac{H(t)}{\Phi^{-1}(t\,M(t))}.
\]
Using the upper bound derived from our induction hypothesis this becomes
\[
\limsup_{t\to\infty}\frac{x(t)}{F^{-1}(t\,M(t))} \leq (1+\epsilon)[(\zeta_n+\epsilon)^\beta + \epsilon]\limsup_{t\to\infty}\frac{\int_{T_2}^t M(t-s) \phi(\Phi^{-1}(s\,M(s)))}{\Phi^{-1}(t\,M(t))} + \lambda.
\]
Applying Karamata's Theorem and Lemma \ref{Beta_Lemma}
\begin{align*}
\limsup_{t\to\infty}\frac{x(t)}{F^{-1}(t\,M(t))} &\leq (1+\epsilon)[(\zeta_n+\epsilon)^\beta + \epsilon]\limsup_{t\to\infty}\frac{\int_{T_2}^t M(t-s) \phi(\Phi^{-1}(s\,M(s)))}{(1-\beta)t\,M(t)\,\phi(\Phi^{-1}(t\,M(t)))} + \lambda\\
&= \frac{(1+\epsilon)[(\zeta_n+\epsilon)^\beta + \epsilon]}{1-\beta}B\left(1+\theta,\frac{1+\theta\beta}{1-\beta}\right)+\lambda.
\end{align*}
Letting $\epsilon\to 0^+$ yields
\[
\limsup_{t\to\infty}\frac{x(t)}{F^{-1}(t\,M(t))} \leq \frac{\zeta^\beta}{1-\beta} B\left(1+\theta,\frac{1+\theta\beta}{1-\beta}\right)+\lambda = \zeta_{n+1},
\]
proving the induction hypothesis $(H_{n+1})$. Hence $(H_n)$ holds for all $n$, or
Hence
\[
\limsup_{t\to\infty}\frac{x(t)}{F^{-1}(t\,M(t))} \leq \zeta_n,  \mbox{ for all } n \geq 0. 
\]
By Lemma \ref{iteration}, $\lim_{n\to\infty}\zeta_n = \zeta$, where $\zeta$ is the unique solution in $[L,U]$ of the ``characteristic'' equation \eqref{characteristic}. Thus
\[
\limsup_{t\to\infty}\frac{x(t)}{F^{-1}(Mt)} \leq \zeta. 
\] 
In the case of the corresponding limit inferior the only modification is to the induction hypothesis, take $\zeta_0:= L$, and the argument then proceeds as above to yield
$
\liminf_{t\to\infty}x(t)/F^{-1}(Mt) \geq \zeta, 
$
completing the proof. 

Now suppose that $(i.)$ holds, or that $\lim_{t\to\infty}x(t)/F^{-1}(t\,M(t)) = \zeta \in [L,\infty)$. It follows that there exists $T_3(\epsilon)>0$ such that for all $t \geq T_3$,
\[
\phi((\zeta-\epsilon)\Phi^{-1}(t\,M(t))) < \phi(x(t)) < \phi((\zeta+\epsilon)\Phi^{-1}(t\,M(t))).
\]
Hence for $t\geq T_3$
\[
\int_{T_3}^t M(t-s) \phi((\zeta-\epsilon)\Phi^{-1}(s\,M(s)))ds \leq \int_{T_3}^t M(t-s) \phi(x(s))ds \leq \int_{T_3}^t M(t-s)\phi((\zeta+\epsilon)\Phi^{-1}(s\,M(s)))ds.
\]
Using the regular variation of $\phi$ the above estimate can be reformulated as
\begin{align*}
\frac{(\zeta-\epsilon)^\beta\,\int_{T_3}^t M(t-s) \phi(\Phi^{-1}(s\,M(s)))\,ds}{\Phi^{-1}(t\,M(t))} &\leq \frac{\int_{T_3}^t M(t-s) \phi(x(s))\,ds}{\Phi^{-1}(t\,M(t))}\\ &\leq (\zeta+\epsilon)^\beta\frac{\int_{T_3}^t M(t-s)\phi((\zeta+\epsilon)\Phi^{-1}(s\,M(s)))\,ds}{\Phi^{-1}(t\,M(t))}, \quad t \geq T_3.
\end{align*}
Using Lemma \ref{Beta_Lemma} and letting $\epsilon\to 0^+$ thus yields
\[
\lim_{t\to\infty}\frac{\int_{0}^t M(t-s) \phi(x(s))ds}{\Phi^{-1}(t\,M(t))} = \frac{\zeta^\beta}{1-\beta}B\left(1+\theta,\frac{1+\theta\beta}{1-\beta}\right).
\]
Therefore assuming $(i.)$ and taking the limit across \eqref{volterra_pert} we obtain
\[
\zeta = \frac{\zeta^\beta}{1-\beta}B\left(1+\theta,\frac{1+\theta\beta}{1-\beta}\right) + \lim_{t\to\infty}\frac{H(t)}{\Phi^{-1}(t\,M(t))},
\]
as claimed.
\end{proof}
We now give the proof of Theorem~\ref{thm.bigpert} in which the perturbation is large. The reader will note that this proof makes much less use of properties of regular varying functions: in fact, we establish the asymptotic result by 
observing that a key functional of the solution is well approximated by a linear non--autonomous differential inequality.
\begin{proof}[Proof of Theorem \ref{thm.bigpert}]
As always $\epsilon\in(0,1)$ is arbitrary. From \eqref{f_monotone_approx} there exists a $\phi$ such that 
\[
\lim_{x \to\infty}f(x)/\phi(x)=1, \quad \lim_{x\to\infty}x\,\phi'(x)/\phi(x) = \beta
\] 
(see e.g.,~\cite[Theorem 1.3.3]{BGT}). Therefore there exists $x_1(\epsilon)>0$ such that $f(x)<(1+\epsilon)\phi(x)$ for all $x \geq x_1(\epsilon)$ and $x_0(\epsilon)$ such that $\phi'(x)<(\beta+\epsilon)\phi(x)/x$ for all $x \geq x_0(\epsilon)$. Similarly, since $\lim_{t\to\infty}x(t)=\infty$, there exists $T_1(\epsilon)>0$ such that $x(t) > \max(x_0(\epsilon),x_1(\epsilon))$ for all $t \geq T_1(\epsilon)$. The regular variation of $M$ means that there exists a non-decreasing function $M_1 \in C^1$ and $T_2(\epsilon)>0$ such that $(1-\epsilon)M_1(t)<M(t)<(1+\epsilon)M_1(t)$ for all $t \geq T_2(\epsilon)$. Hence
\[
(1-\epsilon)M_1(t)< \max_{T_2\leq s \leq t}M(s)<(1+\epsilon)M_1(t),\quad t \geq T_2.
\]
Thus for $t \geq T_2$
\[
(1-\epsilon)M_1(t)< \max_{0\leq s \leq t}M(s)< \max\left(\max_{0\leq s \leq T_2}M(s), \max_{T_2 \leq s \leq t}M(s)\right) \leq \max\left(\max_{0\leq s \leq T_2}M(s), (1+\epsilon)M_1(t)\right).
\]
Therefore
\[
1-\epsilon \leq \frac{\max_{0\leq s \leq t}M(s)}{M_1(t)} \leq \max\left(\frac{\max_{0\leq s \leq T_2}M(s)}{M_1(t)}, 1+\epsilon\right),
\]
and because $\lim_{t\to\infty}M_1(t)=\infty$ we conclude that $\lim_{t\to\infty}\max_{0\leq s \leq t}M(s)/M_1(t)=1$. It follows that there exists a $T_3(\epsilon)>0$ such that $\max_{0\leq s \leq t}M(s) < (1+\epsilon)M_1(t)$ for all $t \geq T_3(\epsilon)$. Now let $T = 1 + \max(T_1,T_2,T_3)$. From \eqref{volterra_pert}, with $t \geq 2T$,
\begin{align*}
x(t) &= x(0) + H(t) + \int_0^T M(t-s)f(x(s))ds + \int_T^{t} M(t-s)f(x(s))\,ds\\
&< x(0) + H(t) + \int_0^T M(t-s)f(x(s))ds + (1+\epsilon)\int_T^{t} M(t-s)\phi(x(s))\,ds\\
&= x(0) + H(t) + \int_0^T M(t-s)f(x(s))ds + (1+\epsilon)\int_T^{t-T} M(t-s)\phi(x(s))\,ds\\
&+ (1+\epsilon)\int_{t-T}^t M(t-s)\phi(x(s))\,ds.
\end{align*}
If $s \in [T,t-T]$, then $t-s \geq T>T_1$, and for $t \geq 2T$
\begin{align*}
x(t)&< x(0) + H(t) + \int_0^T M(t-s)f(x(s))\,ds + (1+\epsilon)^2 M_1(t) \int_T^{t-T} \phi(x(s))\,ds\\
&+ (1+\epsilon)\max_{0\leq s \leq T}M(s)\int_{t-T}^t\phi(x(s))\,ds
\end{align*}
Now, as $T > T_3(\epsilon)$, $\max_{0 \leq s \leq T}M(s) < (1+\epsilon)M_1(T) < (1+\epsilon)M_1(t)$. Hence
\begin{align*}
x(t)&< x(0) + H(t) + \int_0^T M(t-s)f(x(s))\,ds + (1+\epsilon)^2 M_1(t) \int_T^{t} \phi(x(s))\,ds, \quad t \geq 2T.
\end{align*}
For $t \geq 2T >T$, $\max_{0 \leq s \leq T}M(t-s) = \max_{t-T \leq u \leq t}M(u) \leq \max_{0 \leq u \leq t}M(u)<(1+\epsilon)M_1(t)$. Thus, for $t \geq 2T$,
\begin{align}\label{upper.est.pert1}
x(t)&< x(0) + H(t) + (1+\epsilon)M_1(t)\int_0^T f(x(s))\,ds + (1+\epsilon)^2 M_1(t) \int_T^{t} \phi(x(s))\,ds.
\end{align}
For $t \in [T,2T]$, $x(t) \leq \max_{s \in [0,2T]}x(s) := x^*_1(\epsilon)$. Combining this with \eqref{upper.est.pert1}
\begin{align}\label{upper.est.pert2}
x(t)&< x^*_1(\epsilon) + H(t) + (1+\epsilon)M_1(t)x^*_2(\epsilon) + (1+\epsilon)^2 M_1(t) \int_T^{t} \phi(x(s))\,ds, \quad t \geq 2T,
\end{align}
where $x^*_2(\epsilon):=\int_0^T f(x(s))ds$. Define for $t \geq 2T$
\begin{align}\label{H_epsilon}
H_\epsilon(t) := x^*_1(\epsilon) + H(t) + (1+\epsilon)M_1(t)x^*_2(\epsilon).
\end{align}
Note that by construction $\lim_{t\to\infty}H_\epsilon(t)/H(t) =1$. Consolidating \eqref{upper.est.pert2} and \eqref{H_epsilon} we have
\begin{align}\label{upper.est.pert3}
x(t)&< H_\epsilon(t) + (1+\epsilon)^2 M_1(t) \int_T^{t} \phi(x(s))\,ds, \quad t \geq 2T.
\end{align}
By defining 
\begin{align*}
I_\epsilon(t) := \int_T^{t} \phi(x(s))\,ds, \quad t \geq 2T.
\end{align*}
we can formulate an advantageous auxiliary differential inequality as follows. Since $x$ is continuous and $\phi \in C^1(0,\infty)$, $I'_\epsilon(t) = \phi(x(t)), \, t \geq 2T$. Moreover, $\lim_{t\to\infty}I_\epsilon(t)=\infty$. By \eqref{upper.est.pert3}
\begin{align}\label{I_upper_est}
I'_\epsilon(t) = \phi(x(t)) < \phi\left( H_\epsilon(t) + (1+\epsilon)^2 M_1(t) I_\epsilon(t) \right), \quad t \geq 2T.
\end{align}
By the Mean Value Theorem, for each $t \geq 2T$, there exists $\xi_\epsilon(t)\in [0,1]$ such that 
\[
\phi\left( H_\epsilon(t) + (1+\epsilon)^2 M_1(t) I_\epsilon(t) \right) = \phi(H_\epsilon)+ \phi'\left( H_\epsilon(t) + \xi_\epsilon(t)(1+\epsilon)^2 M_1(t) I_\epsilon(t)\right)(1+\epsilon)^2 M_1(t) I_\epsilon(t).
\]
Let
$
a_\epsilon(t) := H_\epsilon(t) + \xi_\epsilon(t)(1+\epsilon)^2 M_1(t) I_\epsilon(t), \,\, t \geq 2T.
$
 For $t \geq 2T$, 
\[
a_\epsilon(t) \geq H_\epsilon(t)> x^*_1(\epsilon) := \max_{s \in [0,2T]}x(s) > x_0(\epsilon).
\]
Therefore, with $\psi \in \text{RV}_\infty(\beta-1)$ a decreasing function asymptotic to $\phi(x)/x$,
\[
\phi'(a_\epsilon(t))< (\beta+\epsilon)\frac{\phi(a_\epsilon(t))}{a_\epsilon(t}< (\beta+\epsilon)(1+\epsilon)\psi(a_\epsilon(t)) < (\beta+\epsilon)(1+\epsilon)\psi(H_\epsilon(t)),\quad t \geq 2T.
\]
But since $\psi(x) \sim \phi(x)/x$ we also have $\psi(H_\epsilon(t))/(1+\epsilon)< \phi(H_\epsilon(t))/H_\epsilon(t)$ and hence
\[
\phi'(a_\epsilon(t))<(\beta+\epsilon)(1+\epsilon)^2 \frac{\phi(H_\epsilon(t))}{H_\epsilon(t)},\quad t \geq 2T.
\]
Combining this estimate with \eqref{I_upper_est} yields
\begin{align*}
I'_\epsilon(t) < \phi(H_\epsilon(t)) + (\beta+\epsilon)(1+\epsilon)^4 \frac{\phi(H_\epsilon(t))}{H_\epsilon(t)}M_1(t) I_\epsilon(t), \quad t \geq 2T.
\end{align*}
Letting $\alpha_\epsilon(t) = (\beta+\epsilon)(1+\epsilon)^4 M_1(t)\, \phi(H_\epsilon(t))/H_\epsilon(t)$, this becomes
$
I'_\epsilon(t) < \phi(H_\epsilon(t)) + \alpha_\epsilon(t)\, I_\epsilon(t)\text{ for } t \geq 2T.
$
Thus the variation of constants formula yields
\[
I_\epsilon(t) \leq e^{\int_T^t \alpha_\epsilon(s)ds}\int_T^t e^{-\int_T^s \alpha_\epsilon(u)du}\phi(H_\epsilon(s))ds,\quad t \geq 2T. 
\]
We reformulate this as
\begin{align}\label{Upper_final}
\frac{I_\epsilon(t)}{\int_T^t \phi(H_\epsilon(s))ds} \leq \frac{\int_T^t e^{-\int_T^s \alpha_\epsilon(u)du}\phi(H_\epsilon(s))ds}{e^{-\int_T^t \alpha_\epsilon(s)ds}\int_T^t\phi(H_\epsilon(s))ds} =: \frac{C_\epsilon(t)}{B_\epsilon(t)}, \quad t \geq 2T.
\end{align}
Since $C_\epsilon'(t) = \phi(H_\epsilon(t))e^{-\int_T^t \alpha_\epsilon(u)du}>0$, we have $\lim_{t\to\infty}C_\epsilon(t) = C^*(\epsilon) \in (0,\infty)$ or $\lim_{t\to\infty}C_\epsilon(t) = \infty$. Also, for $t \geq 2T$,
\begin{align*}
B_\epsilon'(t) &= \phi(H_\epsilon(t))e^{-\int_T^t \alpha_\epsilon(u)du} - \alpha_\epsilon(t)e^{-\int_T^t \alpha_\epsilon(u)du}\int_T^t \phi(H_\epsilon(s))ds\\
&= C_\epsilon'(t) - \frac{\alpha_\epsilon(t)\, C_\epsilon'(t)\, \int_T^t \phi(H_\epsilon(s))ds }{\phi(H_\epsilon(t))}
=  C_\epsilon'(t)\left\{1 - \frac{\alpha_\epsilon(t)\, \int_T^t \phi(H_\epsilon(s))ds }{\phi(H_\epsilon(t))} \right\}.
\end{align*}
Therefore, recalling the definition of $\alpha_\epsilon(t)$ and rearranging,
\[
\frac{B_\epsilon'(t)}{C_\epsilon'(t)} = 1 - (\beta+\epsilon)(1+\epsilon)^4 \left( \frac{M_1(t) \int_T^t \phi(H_\epsilon(s))ds}{H_\epsilon(t)} \right), \quad t \geq 2T.
\]
Letting $t \to \infty$ and using the hypothesis \eqref{bigpert_hypth}, and that $H_\epsilon(t) \sim H(t)$ and $M_1(t) \sim M(t)$ as $t\to\infty$,  yields $\lim_{t\to\infty}B_\epsilon'(t)/C_\epsilon'(t) = 1$, or that $\lim_{t\to\infty}C_\epsilon'(t)/B_\epsilon'(t) = 1$. Hence there exists $T_4$ such that $B_\epsilon'(t)>0$ $t \geq T_4$ and either $\lim_{t\to\infty}B_\epsilon(t) = B^*(\epsilon) \in (0,\infty)$ or $\lim_{t\to\infty}B_\epsilon(t) = \infty$. Furthermore, asymptotic integration shows that $\lim_{t\to\infty}C_\epsilon(t)=\infty$ implies $\lim_{t\to\infty}B_\epsilon(t)=\infty$ and $\lim_{t\to\infty}C_\epsilon(t)=C^*(\epsilon)$ implies $\lim_{t\to\infty}B_\epsilon(t)=B^*(\epsilon)$. Hence,
\[
\Lambda(\epsilon):= \lim_{t\to\infty}\frac{C_\epsilon(t)}{B_\epsilon(t)} = 
\begin{cases}
1, \,\, \lim_{t\to\infty}C_\epsilon(t)=\infty,\\
\tfrac{C*(\epsilon)}{B^*(\epsilon)},\,\,\lim_{t\to\infty}C_\epsilon(t)=C^*,
\end{cases}
\]
where the first limit is calculated using L'H\^{o}pital's rule. Taking the limit superior across equation \eqref{Upper_final} then yields
\begin{align}\label{epsilon_dependent_upper1}
\limsup_{t\to\infty}\frac{\int_T^t \phi(x(s))ds}{\int_T^t \phi(H_\epsilon(s))ds} = \limsup_{t\to\infty}\frac{I_\epsilon(t)}{\int_T^t \phi(H_\epsilon(s))ds} \leq \Lambda(\epsilon) \in (0,\infty).
\end{align}
Since $H_\epsilon(t) \sim H(t)$ as $t\to\infty$ and $\phi$ is increasing we can apply L'H\^{o}pital's rule once more to compute
\[
\lim_{t\to\infty}\frac{\int_T^t \phi(H_\epsilon(s))ds}{\int_0^t \phi(H_\epsilon(s))ds} = \lim_{t\to\infty}\frac{\phi(H_\epsilon(t))}{\phi(H(t))} = 1^\beta = 1,
\]
using that $\phi \in \text{RV}_\infty(\beta)$. A similar argument relying on the divergence of $\phi(x(t))$ and L'H\^{o}pital's rule yields $\int_T^t \phi(x(s))ds \sim \int_0^t \phi(x(s))ds$ as $t\to\infty$. Therefore \eqref{epsilon_dependent_upper1} is equivalent to 
\begin{align}\label{epsilon_dependent_upper2}
\limsup_{t\to\infty}\frac{\int_0^t \phi(x(s))ds}{\int_0^t \phi(H(s))ds}  \leq \Lambda(\epsilon) \in (0,\infty).
\end{align}
Therefore there exists a $\Lambda^* \in (0,\infty)$ such that 
$
\limsup_{t\to\infty}\int_0^t \phi(x(s))ds/\int_0^t \phi(H(s))ds  \leq \Lambda^*,
$
with $\Lambda^*$ independent of $\epsilon$. Thus there exists a $T_6(\epsilon)$ such that $\int_0^t \phi(x(s))ds < (\Lambda^*+\epsilon)\int_0^t \phi(H(s))ds$ for all $t \geq T_6(\epsilon)$. Letting $\bar{T} = 1 + \max(2T,T_6)$ we apply this estimate to \eqref{upper.est.pert3} as follows
\begin{align*}
\frac{x(t)}{H(t)} &< \frac{H_\epsilon(t)}{H(t)} + \frac{(1+\epsilon)^2 \, M_1(t)\, \int_T^t \phi(x(s))ds}{H(t)}< \frac{H_\epsilon(t)}{H(t)} + \frac{(1+\epsilon)^2 \, M_1(t)\, (\Lambda^*+\epsilon)\int_0^t \phi(H(s))ds}{H(t)}, \quad t \geq \bar{T}.
\end{align*}
Now, since $H_\epsilon(t) \sim H(t)$ as $t\to\infty$ and $M_1 \sim M$, applying \eqref{bigpert_hypth} to the above estimate yields
$
\limsup_{t\to\infty}x(t) / H(t) \leq 1.
$
By positivity \eqref{volterra_pert} admits the trivial bound $x(t) > H(t)$ for all $t \geq 0$ and hence $\liminf_{t\to\infty}x(t)/H(t) \geq 1$, completing the proof.
\end{proof}
\section{Proofs of Miscellaneous Propositions and Examples}\label{sec.examples}
\begin{proof}[Proof of Proposition \ref{interpret}]
$(i.)$ is clear from inspection. For $(ii.)$ recall the following form of Sterling's approximation (see \cite[eq. 5.11.7, p.141]{nist})
\begin{align}\label{sterling}
\Gamma(az+b) \sim \sqrt{2\pi}e^{-az}(az)^{az+b - \tfrac{1}{2}}, \text{ as }z\to\infty, \text{ for }a \in (0,\infty),\,\,b \in \mathbb{R}.
\end{align}
Hence, as $\theta\to\infty$,
\begin{align*}
\Lambda(\beta,\theta) &\sim \frac{\left(\sqrt{2\pi}\,e^{-\theta}\theta^{\theta+\tfrac{1}{2}}\right) \left(\sqrt{2\pi}\,e^{\frac{-\beta\theta}{1-\beta}}\{\frac{\beta\theta}{1-\beta}\}^{\frac{\beta\theta}{1-\beta}+\frac{1}{1-\beta}-\frac{1}{2}}\right)}{\left(\sqrt{2\pi}\,e^{\frac{\theta}{\beta-1}}\{\frac{\theta}{1-\beta}\}^{\frac{\theta}{1-\beta}+\frac{1}{1-\beta}-\frac{1}{2}}\right)} \sim \sqrt{2\pi}\,\theta^{\frac{1}{2}}\left((1-\beta)\beta^{\frac{\beta}{1-\beta}}\right)^\theta \,\beta^{\frac{1}{1-\beta}-\frac{1}{2}}.
\end{align*}
Therefore, since $(1-\beta)\beta^{\frac{\beta}{1-\beta}} \in (0,1)$ for $\beta \in (0,1)$, $\lim_{\theta\to\infty}\Lambda(\beta,\theta)=0$, for each fixed $\beta \in (0,1)$. Now, for fixed $\theta \in(0,\infty)$, let $z(\beta)=(\beta\theta+1)/(\theta(1-\beta))$ and note that $\lim_{\beta\uparrow 1}z(\beta)=\infty$. Applying \eqref{sterling} then yields
\begin{align*}
\Lambda(\beta,\theta) &= \frac{\Gamma(1+\theta)\Gamma\left( \theta z(\beta)\right)}{\Gamma(\theta z(\beta) + \theta)} \sim \Gamma(1+\theta)(\theta z(\beta))^{-\theta} \sim \Gamma(1+\theta)\left(\frac{1+\beta\theta}{1-\beta}\right)^{-\theta} \mbox{ as } \beta \uparrow 1.
\end{align*}
Therefore $\lim_{\beta\uparrow 1}\Lambda(\beta,\theta) = 0$ for each fixed $\theta \in (0,\infty)$. \\\\*
To see $(iii.)$ compute $\frac{\partial}{\partial \beta}\Lambda(\beta,\theta)$ as follows
\begin{align*}
\frac{\partial}{\partial \beta}\Lambda(\beta,\theta) &=  \frac{(\theta +1) \Gamma(\theta +1)}{(1-\beta)^2}\left\{\frac{\Gamma'\left(\frac{\beta\theta+1}{1-\beta}\right)}{\Gamma\left(\frac{1+\theta}{1-\beta}\right)} 
- \frac{\Gamma'\left(\frac{\theta+1}{1-\beta}\right)\Gamma\left(\frac{\beta\theta+1}{1-\beta}\right)}{\left(\Gamma\left(\frac{\theta+1}{1-\beta}\right)\right)^2}\right\} \\
&=  \frac{(\theta +1) \Gamma(\theta +1)\Gamma\left(\frac{\beta\theta+1}{1-\beta}\right)}{(1-\beta)^2\,\Gamma\left(\frac{\theta+1}{1-\beta}\right)}\left\{\frac{\Gamma'\left(\frac{\beta\theta+1}{1-\beta}\right)}{\Gamma\left(\frac{1+\beta\theta}{1-\beta}\right)} 
- \frac{\Gamma'\left(\frac{\theta+1}{1-\beta}\right)}{\Gamma\left(\frac{\theta+1}{1-\beta}\right)}\right\} \\
&= \frac{-(\theta +1) \Gamma (\theta +1) \Gamma \left(\frac{\beta  \theta +1}{1-\beta }\right) \left\{\psi\left(\frac{\theta +1}{1-\beta
   }\right)-\psi\left(\frac{\beta  \theta +1}{1-\beta }\right)\right\}}{(\beta -1)^2 \Gamma \left(\frac{\theta +1}{1-\beta }\right)},
\end{align*}
where $\psi(x) := \Gamma'(x)/\Gamma(x)$ (as in \cite[eq. 5.2.2, p.137]{nist}). Hence $\frac{\partial}{\partial \beta}\Lambda(\beta,\theta) < 0$ for $\beta \in (0,1)$ and for each fixed $\theta \in (0,\infty)$ if and only if 
\[
\psi\left(\frac{\theta +1}{1-\beta
   }\right)-\psi\left(\frac{\beta  \theta +1}{1-\beta }\right)>0.
\] 
This holds because $\psi$ is monotone increasing on $\mathbb{R}^+$ (see \cite[eq. 5.7.6, p.139]{nist}). Similarly, to prove claim $(iv.)$, it can be shown that 
\[
\frac{\partial}{\partial \theta}\Lambda(\beta,\theta) = \frac{\Gamma (\theta +1) \Gamma \left(\frac{\beta  \theta +1}{1-\beta }\right) \left\{(\beta -1) \psi(\theta +1)+\psi\left(\frac{\theta +1}{1-\beta }\right)-\beta  \psi\left(\frac{\beta  \theta +1}{1-\beta }\right)\right\}}{(\beta -1) \Gamma
   \left(\frac{\theta +1}{1-\beta }\right)}.
\]
Hence $\frac{\partial}{\partial \beta}\Lambda(\beta,\theta) < 0$ for $\theta \in (0,\infty)$ and for each fixed $\beta \in (0,1)$ if and only if 
\begin{align}\label{theta_decr}
(\beta -1) \psi(\theta +1)+\psi\left(\frac{\theta +1}{1-\beta }\right)-\beta  \psi\left(\frac{\beta  \theta +1}{1-\beta }\right) > 0.
\end{align}
By monotonicity of $\psi$, \eqref{theta_decr} is implied by 
\[
(\beta -1) \psi(\theta +1)+\psi\left(\frac{\theta +1}{1-\beta }\right)-\beta  \psi\left(\frac{\theta +1}{1-\beta }\right) > 0.
\]
However, this inequality is equivalent to
\[
\psi\left(\frac{\theta +1}{1-\beta }\right) > \psi(\theta +1).
\]
Since $\psi$ is increasing and $\beta \in (0,1)$ this holds for all $\theta \in (0,\infty)$. Thus \eqref{theta_decr} holds and the claim is proven. Finally, claim $(v.)$ follows from the continuity of $\Lambda$ and claims $(i.)-(iv.)$.
\end{proof}
\begin{proof}[Proof of Proposition \ref{asym_lemma}]
Applying Karamata's Theorem to $F$ yields the first part of the claim. We restate this in terms of $\ell$ as follows
\[
F(x) \sim \frac{1}{1-\beta}\frac{x^{1-\beta}}{\ell(x)^{1-\beta}}, \text{ as }x\to\infty.
\]
Note that $L(x) := 1/(1-\beta)^{\tfrac{1}{1-\beta}}\,\ell(x) \in \text{RV}_\infty(0)$. Hence applying \cite[Theorem 1.5.15]{BGT}
\begin{align}\label{F_inv_asym}
F^{-1}(x) \sim x^{\frac{1}{1-\beta}}\, L^{\#}\left(x^{\frac{1}{1-\beta}}\right),\text{  as }x\to\infty.
\end{align}
where $L^{\#}$ denotes the de Bruijn conjugate of $L$ (see \cite[eq. 1.5.12, p.29]{BGT}). It can be shown that \eqref{debruijn} is equivalent to
$
\lim_{x\to\infty}L(x/L(x))/L(x) = 1,
$
and by \cite[Corollary 2.3.4]{BGT} this is a sufficient condition for $L^{\#}(x) \sim 1/L(x)$ as $x\to\infty$. Combining this fact with \eqref{F_inv_asym} we conclude that
\[
F^{-1}(x) \sim x^{\frac{1}{1-\beta}}\left\{L(x^{\frac{1}{1-\beta}})\right\}^{-1} \sim (1-\beta)^{\frac{1}{1-\beta}}\,\ell\left(x^{\frac{1}{1-\beta}}\right)x^{\frac{1}{1-\beta}}, \text{ as }x\to\infty,
\]
completing the proof.
\end{proof}
\begin{proof}[Proof of Proposition \ref{discrete}]
By hypothesis there exists a monotone decreasing, $C^1$ function $\tilde{\mu}$ such that $\mu_0(s) \sim \tilde{\mu}(s)$ as $s\to\infty$. Define $\tilde{M}(t) := \int_0^t \tilde{\mu}(ds)$ and note that $\tilde{M} \in \text{RV}_\infty(\theta)$. We claim that $M(t) \sim \tilde{M}(t)$ as $t\to\infty$. To see this first note that for an arbitrary $\epsilon\in(0,1)$
$
1-\epsilon < \mu_0(j\tau) / \tilde{\mu}(j\tau)<1+\epsilon, \,\, \mbox{ for all } j\tau \geq J(\epsilon),
$
for some $J(\epsilon)\in \mathbb{Z}^+$. Hence 
\begin{align}\label{sum_est}
(1-\epsilon)\sum_{j=J(\epsilon)}^{\lfloor t/\tau \rfloor}\tilde{\mu}(j\tau) \leq \sum_{j=J(\epsilon)}^{\lfloor t/\tau \rfloor}\mu_0(j\tau) \leq (1+\epsilon)\sum_{j=J(\epsilon)}^{\lfloor t/\tau \rfloor}\tilde{\mu}(j\tau),
\end{align}
for all $t \geq \tau(1+J(\epsilon))$. Noting that $M$ is given by \eqref{M_discrete}, we can write \eqref{sum_est} as
\begin{align}\label{sum_est_2}
M((J(\epsilon)-1)\tau )+ (1-\epsilon)\sum_{j=J(\epsilon)}^{\lfloor t/\tau \rfloor}\tilde{\mu}(j\tau) \leq M(t) \leq (1+\epsilon)\sum_{j=J(\epsilon)}^{\lfloor t/\tau \rfloor}\tilde{\mu}(j\tau) + M((J(\epsilon)-1)\tau ),
\end{align}
for all $t \geq \tau(1+J)$. It follows that 
\begin{align}\label{sum_upper_est}
 M(t) &\leq (1+\epsilon)\sum_{l=J-1}^{\lfloor t/\tau \rfloor -1}\int_{l\tau}^{(l+1)\tau}\tilde{\mu}(s)ds + M((J-1)\tau )\leq (1+\epsilon)\int_{(J-1)\tau}^{\lfloor t/\tau \rfloor\tau + \tau}\tilde{\mu}(s)ds + M((J-1)\tau )\nonumber\\ &\leq (1+\epsilon)\int_{(J-1)\tau}^{t + \tau}\tilde{\mu}(s)ds + M((J-1)\tau ), \quad t \geq \tau(1+J).
\end{align}
Similarly, from \eqref{sum_est_2},
\begin{align}\label{sum_lower_est}
M(t) &\geq M((J-1)\tau )+ (1-\epsilon)\sum_{j=J(\epsilon)}^{\lfloor t/\tau \rfloor}\tilde{\mu}(j\tau)\geq M((J-1)\tau )+ (1-\epsilon)\sum_{j=J(\epsilon)}^{\lfloor t/\tau \rfloor}\int_{j\tau}^{(j+1)\tau}\tilde{\mu}(s)ds \nonumber \\
&\geq  M((J-1)\tau )+ (1-\epsilon)\int_{j\tau}^{\lfloor t/\tau \rfloor \tau+\tau}\tilde{\mu}(s)ds \geq M((J-1)\tau )+ (1-\epsilon)\int_{j\tau}^{\lfloor t/\tau \rfloor \tau}\tilde{\mu}(s)ds,
\end{align}
for $t \geq \tau(1+J)$. Hence combining the upper estimate \eqref{sum_upper_est} and the lower estimate \eqref{sum_lower_est} yields
\begin{align*}
M((J-1)\tau )+ (1-\epsilon)\int_{j\tau}^{\lfloor t/\tau \rfloor \tau}\tilde{\mu}(s)ds \leq M(t) \leq (1+\epsilon)\int_{(J-1)\tau}^{t + \tau}\tilde{\mu}(s)ds + M((J-1)\tau ), \,\, t \geq \tau(1+J).
\end{align*}
Therefore 
$
\lim_{t\to\infty}M(t)/\tilde{M}(t)=1,
$
and $M \in \text{RV}_\infty(\theta)$. 
\end{proof}
\bibliography{references}
\bibliographystyle{abbrv}
\end{document}